\documentclass[11pt]{amsart}
\usepackage{fullpage,xcolor,graphicx}
\usepackage{hyperref}
\usepackage{pifont}
\usepackage{amsthm}
\usepackage{amsfonts}
\usepackage{amssymb}
\usepackage[mathscr]{euscript}
\usepackage[all]{xy}
\usepackage{amsmath}
\usepackage{epsfig}
\usepackage{latexsym}
\usepackage{upgreek}

\newtheorem{theorem}{Theorem}
\newtheorem{lemma}[theorem]{Lemma}

\theoremstyle{definition}
\newtheorem{definition}[theorem]{Definition}

\newtheorem*{acknowledgments}{Acknowledgments}

\theoremstyle{remark}
\newtheorem{remark}[theorem]{Remark}

\newcommand{\lk}{\operatorname{\ell{\it k}}}
\newcommand{\vlk}{\operatorname{{\it v}\ell{\it k}}}

\title[Virtual Seifert Surfaces]{Virtual Seifert Surfaces}

\author[M. Chrisman]{Micah Chrisman}
\address{Department of Mathematics, Monmouth University, West Long Branch, New Jersey}
\email{mchrisma@monmouth.edu}


\begin{document}
\begin{abstract} A virtual knot that has a homologically trivial representative $\mathscr{K}$ in a thickened surface $\Sigma \times [0,1]$ is said to be an almost classical (AC) knot. $\mathscr{K}$ then bounds a Seifert surface $F\subset \Sigma \times [0,1]$. Seifert surfaces of AC knots are useful for computing concordance invariants and slice obstructions. However, Seifert surfaces in $\Sigma \times [0,1]$ are difficult to construct. Here we introduce virtual Seifert surfaces of AC knots. These are planar figures representing $F \subset \Sigma \times [0,1]$. An algorithm for constructing a virtual Seifert surface from a Gauss diagram is given. This is applied to computing signatures and Alexander polynomials of AC knots. A canonical genus of AC knots is also studied. It is shown to be distinct from the virtual canonical genus of Stoimenow-Tchernov-Vdovina.\end{abstract}
\maketitle
For an oriented knot or link $L$ in $S^3$, H. Seifert gave in \cite{seifert} a simple algorithm for constructing a compact oriented surface $F$ bounded by $L$, now called a Seifert surface of $L$. As is well known, Seifert surfaces appear prominently in the computation of many link invariants, such as Alexander polynomials, signature functions, and Milnor invariants. Seifert surfaces can also be used to study homologically trivial knots in other $3$-manifolds (e.g. see \cite{ct,UK}). The construction of a Seifert surface, however, can be both mathematically and artistically challenging in the more general setting. 
\newline
\newline
Here we consider two related questions: (1) how does one draw a Seifert surface for a homologically trivial knot in a thickened surface $\Sigma \times [0,1]$, where $\Sigma$ is compact, connected, and oriented, and (2) how can such surfaces be employed in the study of virtual knots? These issues arise, for example, in the computation of the \emph{directed signature functions} \cite{bcg2} of Boden, Gaudreau, and the author. Directed signatures give bounds on the slice genus for those virtual knots that can be represented by a homologically trivial knot in some thickened surface. To compute them, a choice of Seifert surface must be made. Hence, a systematic means for producing Seifert surfaces in $\Sigma \times [0,1]$ is needed in order to extract the geometric content of directed signature functions.

\begin{figure}[htb]
\begin{tabular}{|c|} \hline \\
\begin{tabular}{c} \def\svgwidth{2.7in}
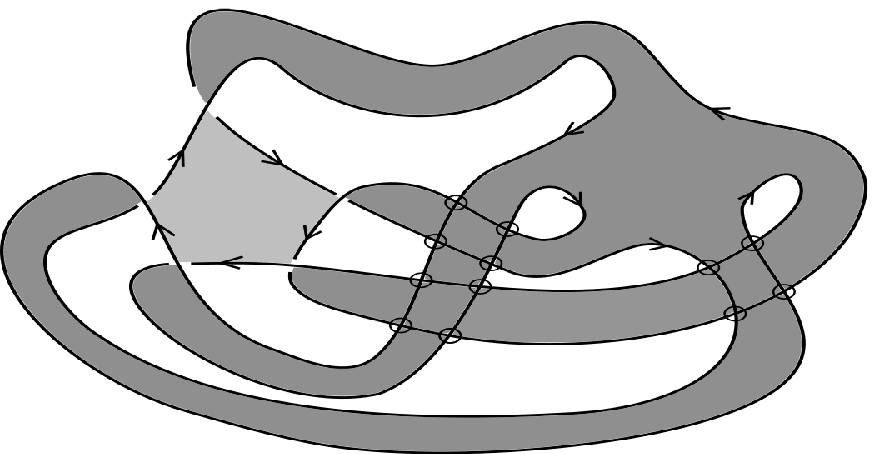 \end{tabular}\\ \hline
\end{tabular}
\end{figure}

A virtual knot that admits a homologically trivial representative in some $\Sigma \times [0,1]$ is called an almost classical (AC) knot. Here we define \emph{virtual Seifert surfaces} for AC knots. A virtual Seifert surface is a planar figure that depicts a Seifert surface $F \subset \Sigma \times [0,1]$, similar to the manner in which a virtual knot diagram is a planar figure depicting a knot $\mathscr{K} \subset \Sigma \times [0,1]$. An example of a virtual Seifert surface is shown above.  Our main result is an algorithm for constructing virtual Seifert surfaces from the Gauss diagram of an AC knot. We apply this to computing slice obstructions. Other applications of virtual Seifert surfaces will appear elsewhere (see \cite{bcg2} v2).
\newline

Recall the Seifert surface algorithm for a classical knot. Starting with a knot diagram $K$ in $S^2=\mathbb{R}^2 \cup \{\infty\}$, perform the oriented smoothing at each crossing. This results in disjoint closed curves called \emph{Seifert cycles}. Each curve bounds a disc in $S^2$, which forms a subsurface of the Seifert surface. The discs are placed at different heights in $S^2 \times [0,1]$ and half-twisted bands are attached to the discs at the crossings of $K$. The result is a Seifert surface $F$ of the knot $K\subset S^2 \times [0,1]$. 
\newline
\newline
The theory of virtual Seifert surfaces given here proceeds along similar lines. For every Gauss diagram $D$, there is a naturally associated surface $\Sigma$ to $D$, called the \emph{Carter surface}. The Gauss diagram $D$ can be drawn as a knot diagram $K$ on $\Sigma$. For an AC knot, $K$ is homologically trivial on $\Sigma$. The homology chain complex of $\Sigma$ and the homology class of each Seifert cycle of $K$ can be computed directly from $D$. Since $[K]=0 \in H_1(\Sigma)$, linear combinations of the Seifert cycles bound subsurfaces of $\Sigma$. The subsurfaces are themselves linear combinations of $2$-handles in the standard handle decomposition of the Carter surface. The linear combinations can be explicitly drawn in the plane, although pieces of the surface may need to pass over one another virtually. The virtual Seifert surface $F$ is again completed by attaching half-twisted bands  at the crossings. 
\newline

A useful feature of classical Seifert surfaces is that they can be deformed into disc-band presentations via pictures in the plane (see e.g. \cite{on_knots}). The advantage of this lies in the fact that the linking numbers in a Seifert matrix are easier to compute when a surface takes this form.  Here we will show how to manipulate virtual Seifert surfaces into \emph{virtual band presentations}. The deformations can also be accomplished through planar pictures. The entries of the Seifert matrix are in this case virtual linking numbers and the calculation of invariants then proceeds as usual.
\newline

The \emph{$3$-genus} of a classical knot is the minimal genus among all Seifert surfaces that it bounds. A Seifert surface constructed from a knot diagram with Seifert's algorithm is said to be \emph{canonical}. The \emph{canonical 3-genus} is the minimal genus among all canonical Seifert surfaces of a knot. Moriah \cite{moriah} proved that the difference between the canonical $3$-genus and $3$-genus of a classical knot can be arbitrarily large. A natural question to ask is whether the virtual Seifert surface algorithm applied to some virtual knot diagram of a classical knot can produce a surface of genus smaller than its canonical $3$-genus. Here we show that this is impossible: the smallest genus among all virtual Seifert surfaces of a classical knot is the classical canonical $3$-genus. This is accomplished by building on important work of Boden-et-al.\cite{acpaper}, Kauffman \cite{lou_cob}, Stoimenow-Tchernov-Vdovina \cite{sto_canon}, and Tchernov \cite{chernov_proj}. As a clarification of terminology, the smallest genus among all virtual Seifert surfaces of an AC knot, which we will call the \emph{virtual canonical 3-genus}, is a fundamentally different concept from the canonical genus defined in \cite{sto_canon}. In fact, we will see an example where they are unequal. For classical knots, we prove the two notions coincide with the classical canonical 3-genus. 
\newline


The organization of this paper is as follows. Section \ref{sec_back} reviews virtual and almost classical knots. Section \ref{sec_homology} discusses a method of computing the homology chain complex of an AC knot. Section \ref{sec_vss_defn} gives precise definitions of virtual band presentations and virtual Seifert surfaces. The virtual Seifert surface algorithm is given in Section \ref{sec_vss}. Section \ref{sec_step_6} shows how to manipulate virtual Seifert surfaces into virtual band presentations. Section \ref{sec_comp} applies virtual Seifert surfaces to computing the slice obstructions from \cite{bcg2}. In Section \ref{sec_canon}, the virtual canonical $3$-genus is studied. In this paper, decimal numbers, such as 5.2025, refer to the virtual knot names from Green's tabulation \cite{green}. The ones digit denotes the classical crossing number. Three examples are used throughout to illustrate the virtual Seifert surface algorithm: 4.99, 5.2025, and 6.87548. Henceforth, we set $I=[0,1]$.

\section{Background}\label{sec_back}
\subsection{Virtual knots} \label{sec_virt_defn} Virtual knots were introduced by L. H. Kauffman in the mid 1990s \cite{KaV}. They have several equivalent definitions. A \emph{virtual knot diagram} is a generic immersion  $\upsilon:S^1 \to \mathbb{R}^2$, where the double points are decorated as classical crossings or virtual crossings. A virtual crossing is denoted by a small circle around the double point. Two virtual knots $\upsilon_1$ and $\upsilon_2$ are said to be equivalent, denoted $\upsilon_1 \leftrightharpoons \upsilon_2$, if one may be obtained from the other by a finite sequence of extended Reidemeister moves (see Figure \ref{fig_reid_moves}). An equivalence class of diagrams is called a \emph{virtual knot}. 

\begin{figure}[htb]
\begin{tabular}{|c|} \hline \\
\begin{tabular}{c} \def\svgwidth{6in}
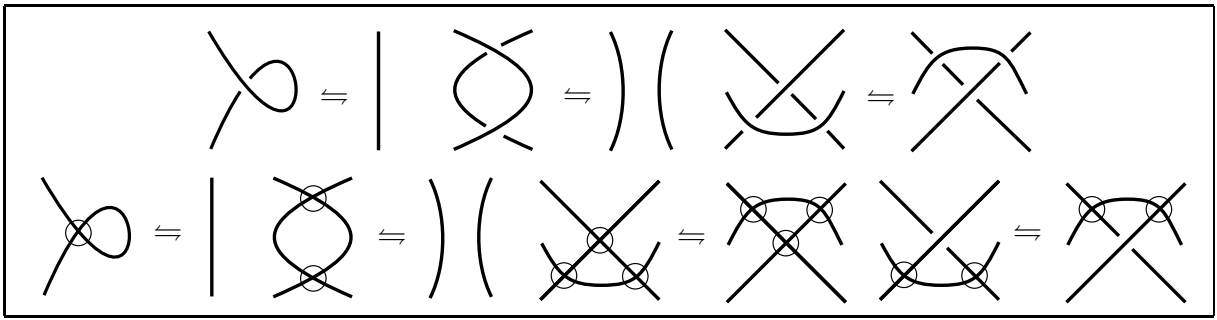 \end{tabular}\\ \hline
\end{tabular}
\caption{The extended Reidemeister moves.} \label{fig_reid_moves}
\end{figure}

The Gauss diagram $D$ of a virtual knot diagram $\upsilon$ is found by connecting the pre-images of the double points of the classical crossings of $\upsilon:S^1 \to \mathbb{R}^2$ by an arrow. The arrow is directed from the over-crossing arc to the under-crossing arc (see Figure \ref{fig_gauss}). Each arrow of $D$ is also marked with the sign $(\pm)$ of the local writhe of the crossing. Any two virtual knot diagrams that have the same Gauss diagram are equivalent \cite{GPV}.
\newline

A virtual knot diagram can then be viewed as the result of an attempt to draw a configuration of classical crossings in a Gauss diagram as a knot diagram in $\mathbb{R}^2$. The inability to do so for a particular configuration necessitates the addition of crossings not specified by the Gauss diagram and hence are marked as virtual.  Every Gauss diagram, however, can be drawn as a knot diagram $K$ on a higher genus surface. One way to see this is to use the \emph{Carter surface} of a virtual knot diagram \cite{CKS}.
\newline

The Carter surface is constructed from a virtual knot diagram $\upsilon$ using a handle decomposition. A $0$-handle (i.e. a disc) is centered at each classical crossing. The $1$-handles are untwisted bands attached to the $0$-handles along the arcs between the classical crossings of $\upsilon$. Thus the arcs of $\upsilon$ are the cores of the 1-handles.  At a virtual crossing, the $1$-handles pass over and under one another, as in Figure \ref{fig_zeroone_hand_attach}. The resulting surface is closed by attaching $2$-handles along each boundary component (see Figure \ref{fig_two_hand_attach}). Observe that the handle decomposition is determined only by the Gauss diagram of $\upsilon$. Furthermore, changing any of classical crossings of $\upsilon$ from over to under (or vice versa) gives a Carter surface with an identical handle decomposition. This fact will be used later.
\newline

\begin{figure}[htb]
\begin{tabular}{|ccc|} \hline & & \\
\begin{tabular}{c} \def\svgwidth{1.35in}
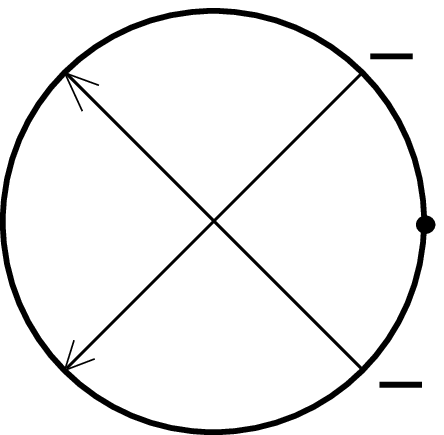 \end{tabular} & $\to$ &  \begin{tabular}{c} \def\svgwidth{1.35in}
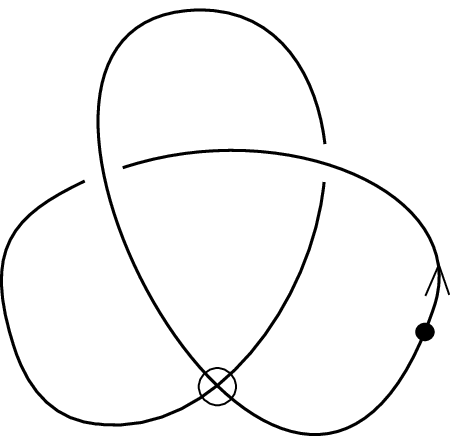 \end{tabular} \\ \hline
\end{tabular}
\caption{A Gauss diagram $D$ (left) and a virtual knot representative of $D$ (right).} \label{fig_gauss}
\end{figure}

\begin{figure}[htb]
\begin{tabular}{|ccc|} \hline & & \\
\begin{tabular}{c} \def\svgwidth{1.35in}
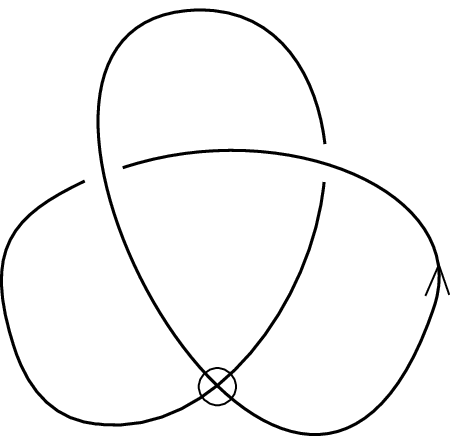 \end{tabular} & 
\begin{tabular}{ccc}
\begin{tabular}{c} \def\svgwidth{.5in}
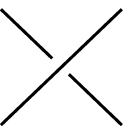 \end{tabular} & $\to$ &  \begin{tabular}{c} \def\svgwidth{.75in}
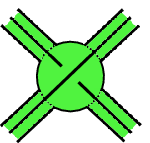 \end{tabular}\\
\begin{tabular}{c} \def\svgwidth{.5in}
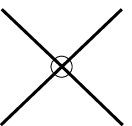 \end{tabular} & $\to$ &  \begin{tabular}{c} \def\svgwidth{.75in}
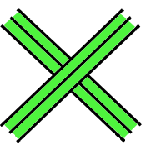 \end{tabular}\\
\end{tabular}
&  \begin{tabular}{c} \def\svgwidth{1.5in}
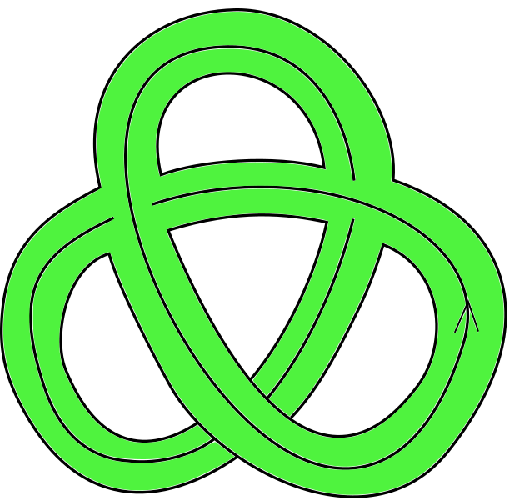 \end{tabular} \\ \hline
\end{tabular}
\caption{Attaching $0$- and $1$-handles of the Carter surface.} \label{fig_zeroone_hand_attach}
\end{figure}

\begin{figure}[htb]
\begin{tabular}{|ccc|} \hline & & \\
\begin{tabular}{c} \def\svgwidth{1.5in}
\input{carter_surf_1.eps_tex} \end{tabular} & $\to$ &  \begin{tabular}{c} \def\svgwidth{2.5in}
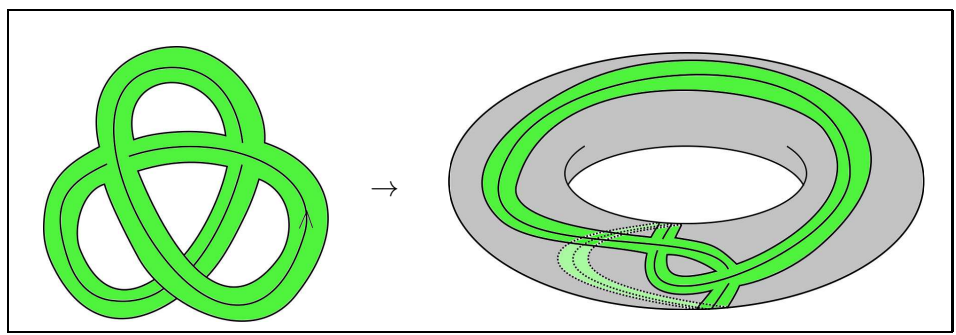 \end{tabular} \\ \hline
\end{tabular}
\caption{Attaching 2-handles to close the Carter surface.} \label{fig_two_hand_attach}
\end{figure}

A Gauss diagram can also be constructed from a knot diagram on any closed oriented surface. If a $1$-handle is added to this surface so that it is disjoint from the knot diagram, then the Gauss diagram of the knot remains unchanged. The operation of adding a $1$-handle is called \emph{stabilization}. The inverse operation is called \emph{destabilization}. The Gauss diagram of a knot diagram on a surface is also unchanged by orientation preserving diffeomorphisms of surfaces. Thus, virtual knots can be interpreted as knots in thickened surfaces, considered equivalent up to ambient isotopy, orientation preserving diffeomorphisms of surfaces, and stabilization/destabilization.  
\newline

The first step in the classical Seifert surface algorithm is to perform the oriented smoothing at every crossing. For future use in constructing virtual Seifert surfaces, we now describe the effect that different kinds of smoothing have on a Gauss diagram. The \emph{oriented smoothing} (see Figure \ref{fig_seif_cross}, far left and far right) at one crossing splits the knot diagram into a two-component oriented link. The \emph{disoriented smoothing} (see Figure \ref{fig_seif_cross}, middle two pictures) at one crossing yields an unoriented knot diagram. To smooth a Gauss diagram, first delete a small neighborhood of each arrow endpoint and then reconnect the ends. The manner of reconnecting for each type of smoothing is also shown in Figure \ref{fig_seif_cross}. Lastly, recall that a smoothing at a crossing may be either an $A$ smoothing or a $B$ smoothing. In an oriented knot diagram, an $A$ or $B$ smoothing may be either oriented or disoriented, depending on the crossing sign. The resulting four possibilities are labeled in Figure \ref{fig_seif_cross}.

\begin{figure}[htb]
\begin{tabular}{cccc}
\begin{tabular}{c}\def\svgwidth{1.3in} 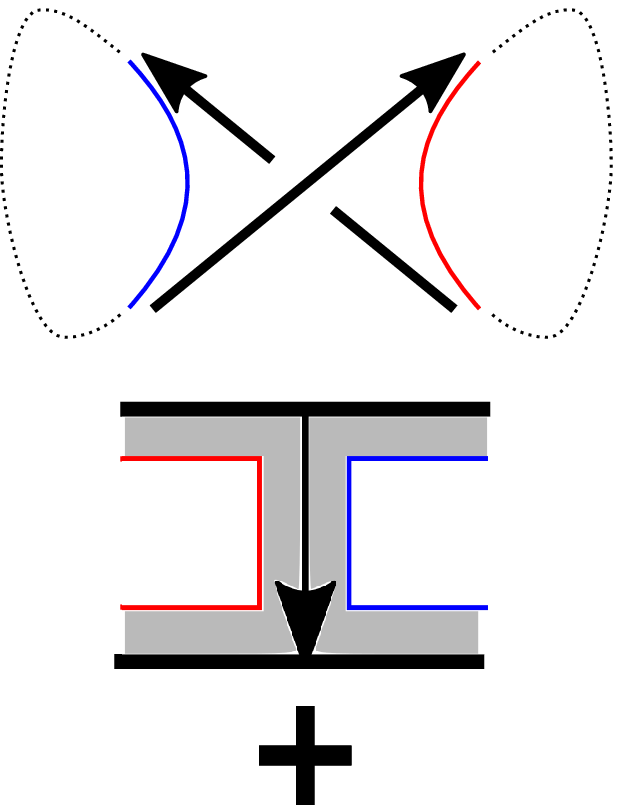\\ \underline{oriented} \end{tabular} & \begin{tabular}{c} \def\svgwidth{1.3in} 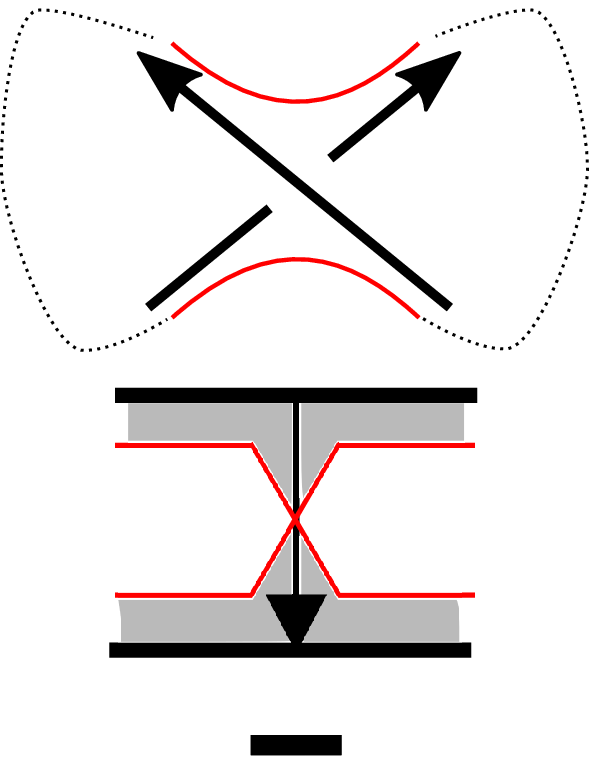 \\ \underline{disoriented} \end{tabular} & \begin{tabular}{c} \def\svgwidth{1.3in} 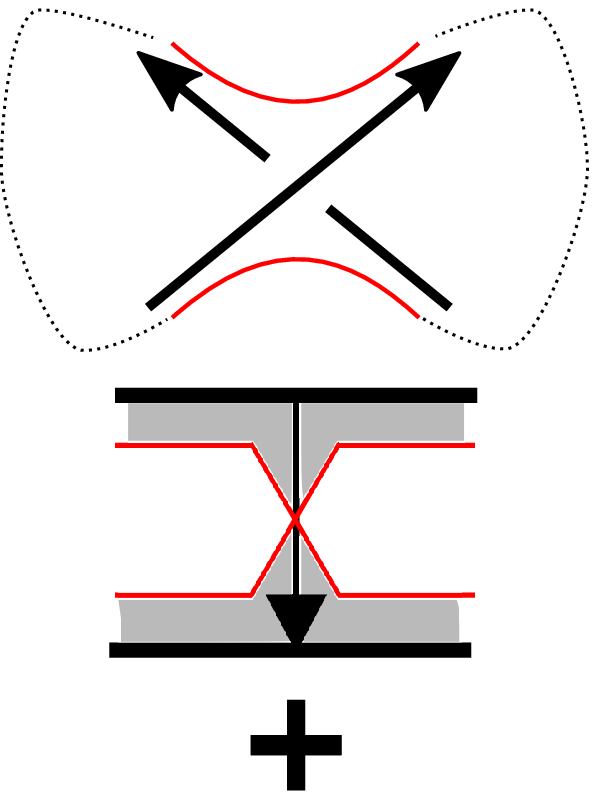 \\ \underline{disoriented} \end{tabular} & \begin{tabular}{c} \def\svgwidth{1.3in} 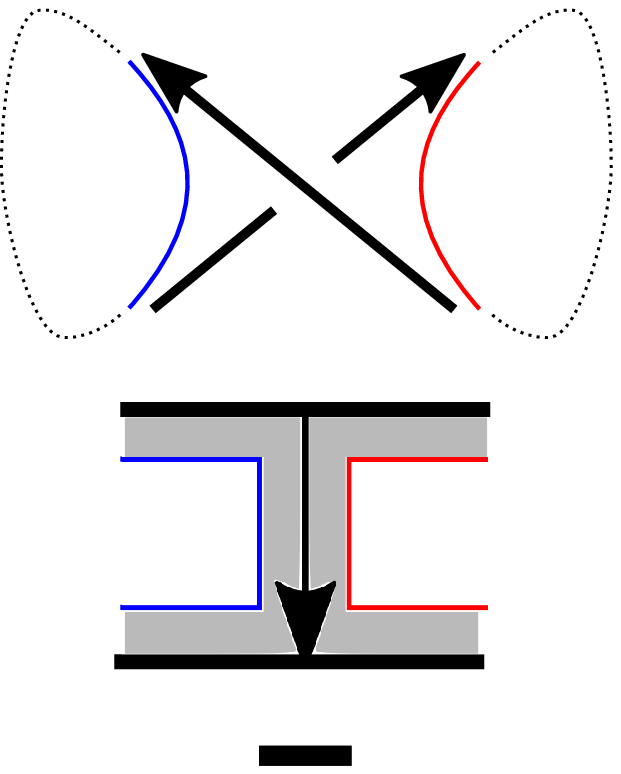 \\ \underline{oriented} \end{tabular}
\end{tabular}
\caption{The effect of smoothing on an arrow in a Gauss diagram.} \label{fig_seif_cross}
\end{figure}

\subsection{Almost Classical Knots} \label{sec_AC} For an oriented knot $\mathscr{K}$ in a $3$-manifold $M$, a \emph{Seifert surface} of $\mathscr{K}$ is a compact connected oriented surface $F \subset M$ such that $\partial F=\mathscr{K}$, where the orientation of $F$ induces the given orientation of $\mathscr{K}$. Suppose $\upsilon$ is a virtual knot, $\mathscr{K}$ is a representative of $\upsilon$ in some thickened surface $\Sigma \times I$, and that $\mathscr{K}$ bounds a Seifert surface $F$. Then $K$ is said to be \emph{almost classical} (AC). Almost classical knots were originally defined by Silver-Williams \cite{silwill} as those virtual knots admitting a diagram with an Alexander numbering. Boden et al.\cite{acpaper} tabulated the distinct AC knots having classical crossing number at most $6$. There are $76$ such nontrivial AC knots in total; their Gauss diagrams can be found in Figure 20 at the end of \cite{acpaper}.
\newline
\newline
There are four equivalent definitions of AC knots that are each useful under different circumstances. For a proof of said equivalence, see Sections 5 and 6 of \cite{acpaper}. The conditions are:
\begin{enumerate}
\item $\upsilon$ admits a diagram with an Alexander numbering.
\item $\upsilon$ has a homologically trivial representative $\mathscr{K}$ in some thickened surface $\Sigma \times I$ (and in particular, admits a diagram  which is homologically trivial on its Carter surface)
\item $\upsilon$ has a representative $\mathscr{K}$ that bounds a Seifert surface $F$ in some thickened surface $\Sigma \times I$.
\item $\upsilon$ admits a Gauss diagram such that every arrow has index $0$. 
\end{enumerate}
Condition (4) provides the simplest method to prove that a given virtual knot is AC. Geometrically, the index of a crossing $x$ of a knot diagram $K$ on a surface $\Sigma$ is (up to sign) the algebraic intersection number of the two closed curves obtained by performing the oriented smoothing at $x$. This can be computed combinatorially from a Gauss diagram (see e.g. \cite{acpaper,henrich}). 
\newline
\newline
By an \emph{AC diagram} (respectively, \emph{AC Gauss diagram}) of a virtual knot, we mean a diagram (respectively, Gauss diagram) that satisfies one of the equivalent conditions for a virtual knot to be AC.
\newline
\newline
A constructive proof that $(2) \implies (3)$ was given in \cite{acpaper}. Here we sketch the argument for later use. Suppose a diagram $K$ of a homologically trivial knot on the Carter surface $\Sigma$ is given. First perform the oriented smoothing at each crossing of $K$. This creates a family of disjoint simple closed curves $s_1,\ldots,s_p$ on $\Sigma$. These are called the \emph{Seifert cycles} of $K$. Since $K$ is homologically trivial on $\Sigma$, so is $s_1 \cup \cdots \cup s_p$. Then there is a collection $S_1,\ldots,S_l$ of connected compact oriented subsurfaces of $\Sigma$ such that $\partial S_i \ne \emptyset$ for all $i$ and $\bigcup_{i=1}^l \partial S_i=\bigcup_{i=1}^p s_i$ (see \cite{acpaper}, Proposition 6.2).  A Seifert surface $F$ is constructed by placing the surfaces $S_1,\ldots,S_l$ at different heights in $\Sigma \times I$ as necessary and then attaching half-twisted bands at the smoothed crossings of $K$. 
\newline
\newline
If $\Sigma=S^2$, this gives the Seifert surface algorithm for classical knots. In this case, each $S_i$ may be chosen to be a disc. This is not true for AC knots in general. Indeed, the genus of each $S_i$ might be any non-negative integer and the number of boundary components of $S_i$ might be any natural number. This leads to an ambiguity in the constructive existence proof for Seifert surfaces from \cite{acpaper}: as there are many options, how does one find linear combinations of the Seifert cycles that bound subsurfaces of $\Sigma$? The virtual Seifert surface algorithm builds off the constructive proof from \cite{acpaper} and reduces this ambiguity to a homological calculation. 

\section{The homology of the Carter surface of AC knots} \label{sec_homology} Here we give a method for computing the homology chain complex of the Carter surface $\Sigma$ for an AC knot diagram $K$ on $\Sigma$. This will be employed in the virtual Seifert surface algorithm. Let $C_i$ be the free abelian group generated by the $i$-handles of the Carter surface (see Section \ref{sec_back}) and $\partial_i:C_i \to C_{i-1}$ the $i$-th boundary map. The 0-handles of $\Sigma$ correspond to the classical crossings of $K$ and hence correspond to the arrows $z_1,\ldots,z_n$ of $D$. The cores of the $1$-handles are the arcs between the classical crossings, and hence correspond to the arcs $c_1,\ldots,c_{2n}$ of $D$ between the arrow endpoints. Denote the $2$-handles of $\Sigma$ by $d_1,\ldots,d_m$. Our main interest is $\partial_2:C_2 \to C_1$. To compute $\partial_2$, it is necessary to write $\partial_2 d_k$ as a linear combination of the $c_j$. First we review some terminology.
\newline
\newline
A virtual knot diagram is said to be \emph{alternating} if the classical crossings alternate successively between over and under while traversing the diagram (note that virtual crossings are ignored). Also recall that the \emph{all $A$ state} (\emph{all $B$ state}) is the set of cycles obtained by performing the $A$ smoothing (respectively, $B$ smoothing) at every classical crossing. The lemma below relates the computation of $\partial_2$ for an AC Gauss diagram $D$ to the all $A$ and all $B$ states of any alternating diagram $D_{\text{alt}}$ that is obtained from $D$ by crossing changes. Recall from Section \ref{sec_virt_defn} that $D_{\text{alt}}$ and $D$ then have identical Carter surfaces. 

\begin{lemma} \label{lemma_AB} Every AC Gauss diagram $D$ can be transformed to an alternating diagram $D_{\text{alt}}$ by a finite number of crossing changes. For any such $D_{\text{alt}}$ there is a one-to-one correspondence between the cycles $\partial_2 d_1,\ldots,\partial_2 d_m \in H_1(\Sigma)$ and the components of the all $A$ state and the all $B$ state.
\end{lemma}
\begin{proof} An AC knot is homologically trivial on its Carter surface. All such virtual knots are checkerboard colorable. N. Kamada proved in \cite{nkam}, Lemma 7, that every checkerboard colorable virtual knot can be transformed to a alternating diagram by a finite number of crossing changes. We remark that $D_{\text{alt}}$ is then found by changing both the direction and sign of the corresponding arrows of $D$. Note that any crossing change on an AC knot diagram does not affect the handle decomposition of the Carter surface. Finally, it follows from Figure \ref{fig_all} that the boundary of each $2$-handle $d_k$ is either a component in the all $A$ state of $D_{\text{alt}}$ or the all $B$ state of $D_{\text{alt}}$.
\end{proof}

\begin{figure}[htb]
\begin{tabular}{c}
\def\svgwidth{3in} 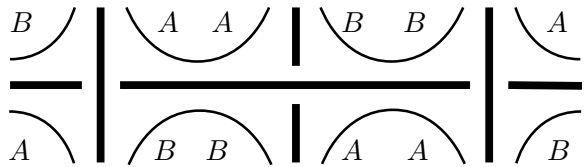
\end{tabular}
\caption{The all $A$ and all $B$ states of an alternating virtual knot diagram.} \label{fig_all}
\end{figure}

We may now write $\partial_2:C_2 \to C_1$ as a matrix in the bases $\{c_1,\ldots,c_{2n}\}$ of $C_1$ and $\{d_1,\ldots,d_m\}$ of $C_2$.  After making $D$ into an alternating diagram $D_{\text{alt}}$, find the components of the all $A$ and all $B$ states. Figure \ref{fig_seif_cross} indicates the effect of the smoothing on a signed arrow. The components are $\partial_2 d_1,\ldots,\partial_2 d_m$. To determine the signs of the contributions of the $c_i$ to $\partial_2 d_k$, we must orient each $\partial_2 d_k$. Draw a classical crossing of $D$ and label it with all of the $1$- and $2$-handles that are incident to it. Give each $2$-handle $d_k$ the counterclockwise orientation. Each edge of $\partial_2 d_k$ is then either $\pm c_i$, according to whether it ``goes with'' $c_i$ ($+$) or ``goes against'' $c_i$ ($-$). Transferring the orientation of the edge back to $D_{\text{alt}}$ indicates the orientation of each edge of $\partial_2 d_k$.  This orients all the $2$-handles incident the chosen crossing in both the all $A$ and all $B$ states. To orient all the other $2$-handles, observe that in each state, the orientations must alternate while traversing the circle of $D_{\text{alt}}$.
\newline
\newline
\textbf{Example (4.99):} A Gauss diagram of $4.99$ is given in Figure \ref{fig_4pt99}. This may be made alternating by changing the direction and sign of the two $+$ signed arrows. The resulting diagram has all $-$ signed arrows. The all $A$ state is then found by choosing the disoriented smoothing at each crossing (see Figure \ref{fig_seif_cross}). The all $B$ state takes, in this case, the oriented smoothing at each crossing. To compute $\partial_2:C_2 \to C_1$, label the eight arcs of $D$ as $c_1,\ldots,c_8$. These $1$-handles generate $C_1$. Similarly, label each of the components of the all $A$ and all $B$ states: $d_1,d_2,d_3,d_4$. Orient the $2$-handles by drawing one crossing and all the $2$-handles incident to it (see Figure \ref{fig_4pt99}, top right). The columns of $\partial_2$ can then be easily read off the figure. We use the symbol $\uptau$ to denote the transpose.
\[                                                                             
\partial_2^{\uptau} = \left[\begin{array}{cccccccc}  1  & -1 & 1 & -1  & 1 & -1 & 1    & -1  \\ 
                                                                               -1 & 0  &  0 & 0 &  -1  & 0 &  0  & 0 \\ 
                                                                               0  & 1  & 0 &  1  &  0 & 1  &  0  &  1   \\ 
                                                                               0  & 0  & -1 &  0  &  0 & 0  &  -1  &  0   \end{array} \right] \begin{array}{c} d_1 \\ d_2 \\ d_3 \\ d_4 \end{array}
\]
Two additional examples of the computation of $\partial_2:C_2 \to C_1$ are given in Figures \ref{fig_5pt2025} and \ref{fig_6pt87548}.
\begin{figure}[htb]
\begin{tabular}{|cc|} \hline & \\
\begin{tabular}{c} 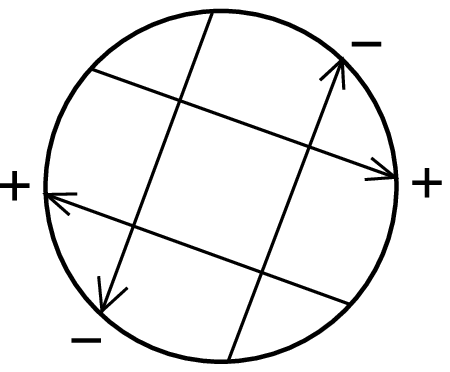 \\ \underline{$D=4.99$} \end{tabular}& \begin{tabular}{|c|} \hline \\ \def\svgwidth{1.5in} 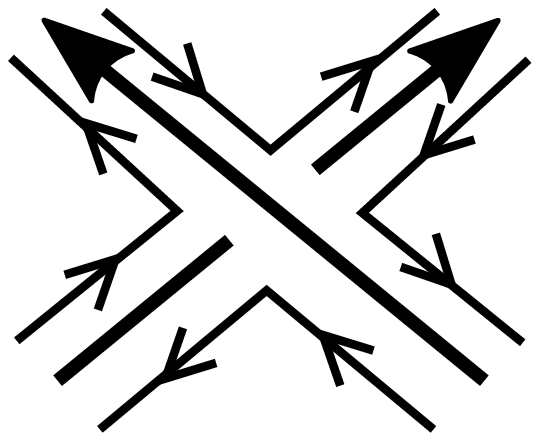 \\ \\ \underline{Orient $2$-handles} \\ \hline \end{tabular} \\
\begin{tabular}{c} 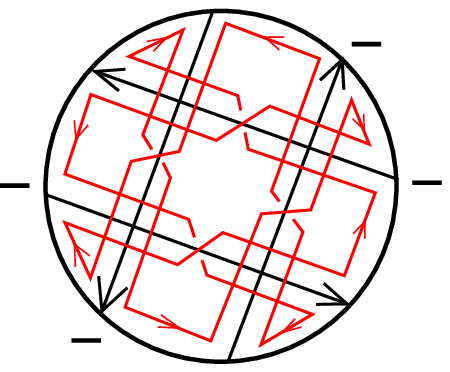 \\ \underline{All $A$ state of $D_{\text{alt}}$} \end{tabular} & \begin{tabular}{c} 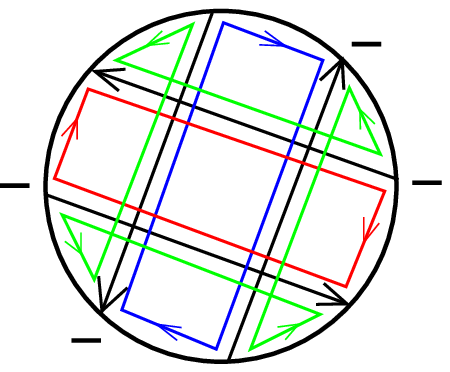 \\ \underline{All $B$ state of $D_{\text{alt}}$} \end{tabular}\\ & \\ \hline
\end{tabular}
\caption{Computing the homology of the Carter surface for $4.99$. After making an alternating diagram, find the all $A$ state and all $B$ state.} \label{fig_4pt99}
\end{figure}

\section{Virtual Band Presentations and Seifert Surfaces}  \label{sec_vss_defn}

To compute a Seifert matrix for a classical knot in $S^3$, one begins with a Seifert surface of the knot and then performs an isotopy to a disc-band presentation. Seifert matrices, and hence invariants such as the Alexander polynomial and the signature, can then be computed from the disc-band presentation. In this section, we give the formal definitions of virtual band presentations and virtual Seifert surfaces. It will be shown that every AC knot has a virtual Seifert surface and every virtual Seifert surface can be deformed to a virtual band presentation. 

\subsection{Virtual Band Presentations} \label{sec_vband} Let $B$ be an oriented disc embedded in $\mathbb{R}^2$. In $\overline{\mathbb{R}^2\smallsetminus B}$, attach smooth arcs $a_1,\ldots,a_n$ to $\partial B$ so that the set of $2n$ points $\partial a_1 \cup \cdots \cup \partial a_n$ are all distinct. Furthermore, we assume that the arcs intersect each other and themselves only in virtual or classical crossings. This can be viewed as a virtual tangle $\tau$ drawn in the disc $\overline{S^2\smallsetminus B}$. Fatten each arc $a_i$ slightly into a immersed band $b_i$ in $\overline{\mathbb{R}^2\smallsetminus B}$ so that there is a classical crossing of bands at each classical crossing and a virtual crossing of bands at each virtual crossing (see Figure \ref{fig_virt_band_cross}). The union $F_{\tau}$ of the disc $B$ together with these bands is called a \emph{virtual band presentation} with underlying virtual tangle $\tau$. The orientation of $B$ determines the orientation of the immersed surface $F_{\tau}$. 
\newline
\newline
Note that $\overline{\left(\partial B \smallsetminus \bigcup \partial b_i\right)} \cup \overline{\left(\left(\bigcup \partial b_i\right) \smallsetminus \partial B\right)}$ constitutes a virtual link diagram $L_{\tau}$, where the natural choices of over, under, and virtual crossings of arcs are made at each double point. We will say $L_{\tau}$ \emph{virtually bounds}  $F_{\tau}$. An example of a virtual band presentation is given in Figure \ref{fig_virt_band_present}. It will be seen later that its virtual boundary is $L_{\tau} \leftrightharpoons 4.99$.
\newline

\begin{figure}[htb]
\begin{tabular}{|cc|} \hline & \\
\def\svgwidth{1.25in}
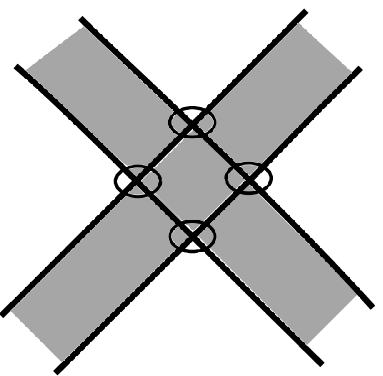 & \def\svgwidth{1.25in}
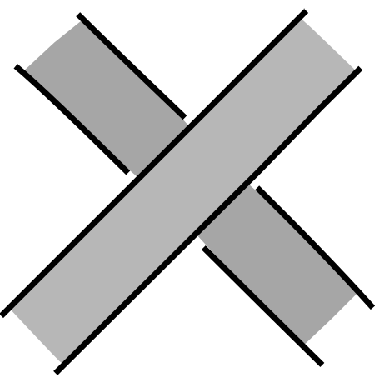 \\ \hline
\end{tabular}
\caption{A virtual crossing of bands (left) and a classical crossing of bands (right).} \label{fig_virt_band_cross}
\end{figure}

\begin{figure}[htb]
\begin{tabular}{|ccc|} \hline & & \\
\def\svgwidth{1.75in}
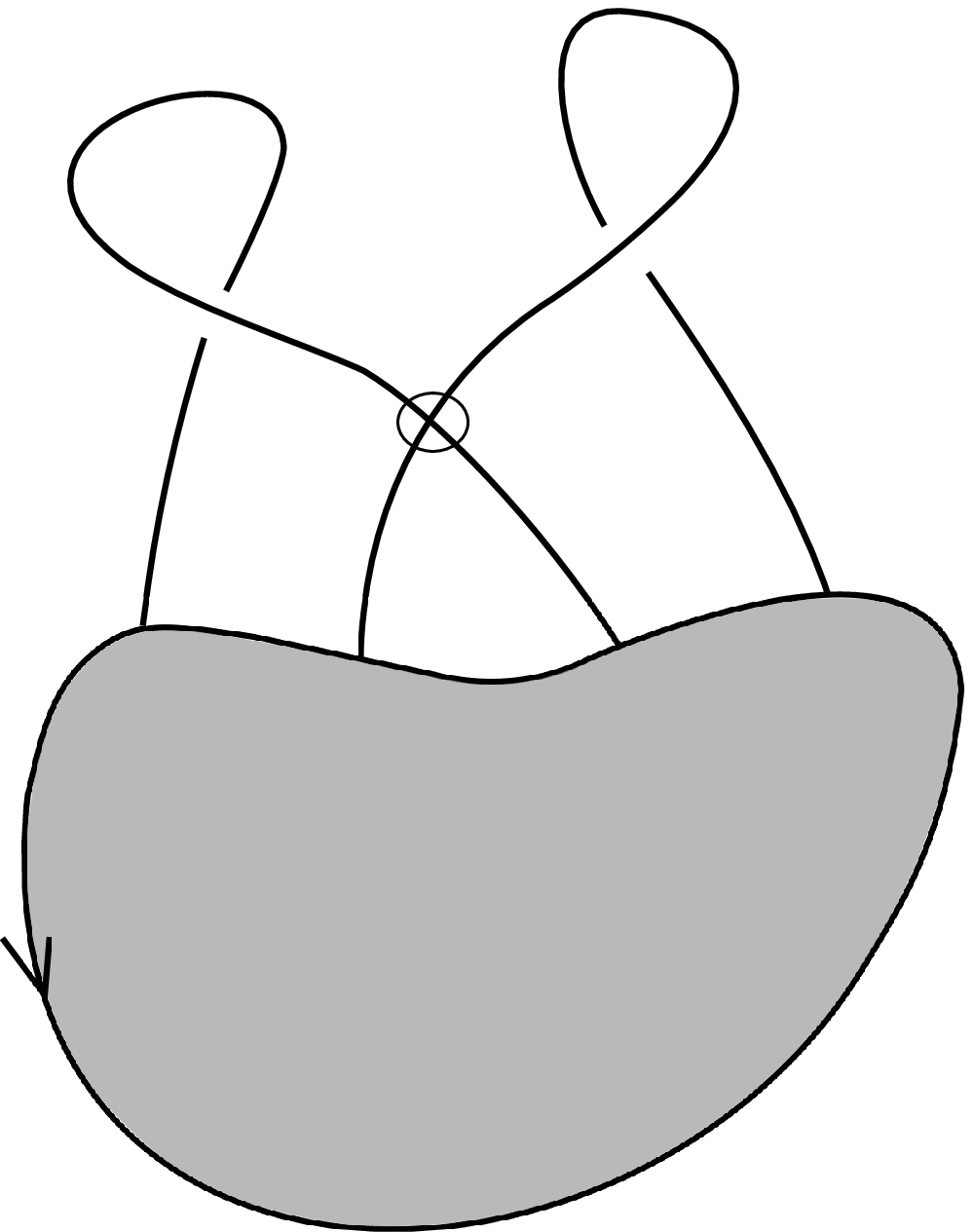 & & \def\svgwidth{1.75in}
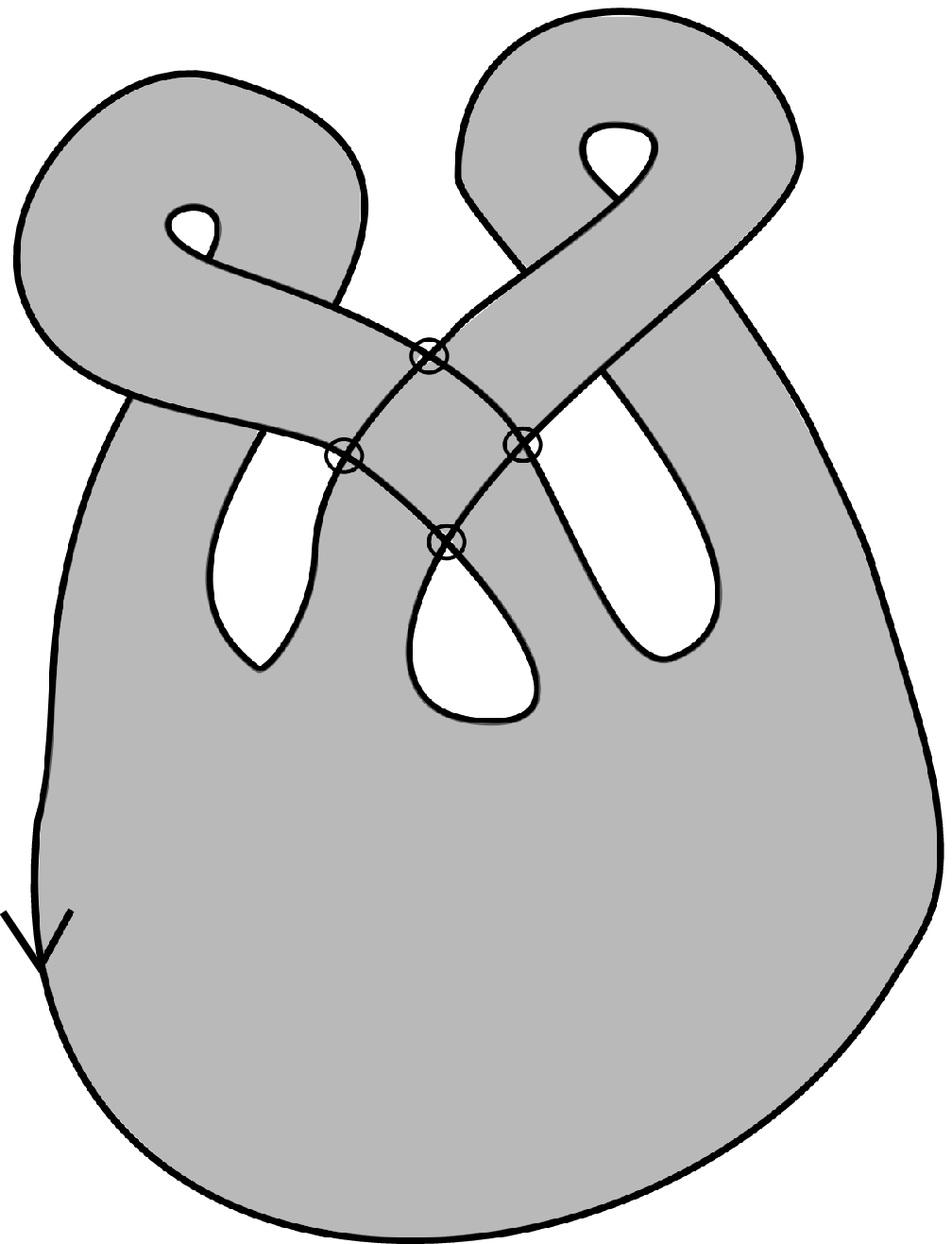 \\ \hline
\end{tabular}
\caption{An virtual tangle (left) and virtual band presentation (right).} \label{fig_virt_band_present}
\end{figure}

Virtual band presentations were introduced in \cite{bcg2} (see v2). They were used to show that if $(V^-,V^+)$ is a pair of integral matrices satisfying $\det(V^--V^+)=1$ and $V^++(V^+)^{\uptau}=V^-+(V^-)^{\uptau}$, then there is an AC knot having $V^{\pm}$ as its pair of directed Seifert matrices (for definitions, see Section \ref{sec_link} ahead). The AC knot that realizes the pair $(V^-,V^+)$ virtually bounds a virtual band presentation. Conversely, our goal is to draw a Seifert surface from an AC Gauss diagram. 

\subsection{Virtual Seifert Surfaces} \label{sec_vss_gen} A useful feature of Seifert surfaces for classical knots is that they can be deformed into disc-band presentations via pictures in the plane. In order to do this in the virtual setting, it is necessary to have a planar representation of the ambient space $\Sigma \times I$. Any compact oriented surface $\Sigma$ with at least one boundary component can be immersed into $\mathbb{R}^2$. The image of the immersion can be thought of intuitively as a screen onto which an actor (e.g. a knot $\mathscr{K}$ or Seifert surface $F$ in $\Sigma \times I$) can be projected. Places where the immersed surface $\Sigma$ overlaps itself are the only places where the images of $\mathscr{K}$ or $F$ can have virtual crossings. A precise definition of \emph{projector} and \emph{screen} is needed to define virtual Seifert surfaces (Definition \ref{defn_vss}) and manipulate them (Section \ref{sec_step_6}).
\newline
\newline
Let $\Xi$ be a compact connected oriented smooth surface with $\partial \,\Xi \ne \emptyset$. Fix a handle decomposition of $\Xi$ into $0$-handles $B_1,\ldots,B_n$ and $1$-handles $O_1,\ldots,O_m$, where each $1$-handle is attached to $\bigcup_{i=1}^n \partial B_i$. Suppose that $\chi:\Xi \to \mathbb{R}^2$ is an orientation preserving immersion that embeds the $B_i$ disjointly in the plane and immerses the $O_j$ in $\overline{\mathbb{R}^2\smallsetminus (\bigcup_{i=1}^n \chi (B_i))}$ such that images of the $1$-handles can intersect themselves and each other only in virtual crossings of bands (see Figure \ref{fig_virt_band_cross}, left). Then a \emph{virtual screen} (or simply \emph{screen}) is the image $X$ of $\chi$. The map $\chi:\Xi \to X$ is a \emph{projector} onto $X$. The overlaps of the immersed $1$-handles are called the \emph{virtual regions}. Each virtual region $R$ is the image of two embedded discs $R' \subset O_i$ and $R'' \subset O_j$ for some $i,j$. In Figure \ref{fig_virt_screen_defn}, $\Xi \subset \mathbb{R}^2 \times \mathbb{R}$ and $\chi$ is the restriction of the canonical projection onto the first factor, i.e. $\mathbb{R}^2 \times \mathbb{R} \to \mathbb{R}^2$.
\newline


\begin{figure}[htb]
\begin{tabular}{|c|} \hline \\
$
\xymatrix{
\begin{array}{c} \def\svgwidth{2.75in}
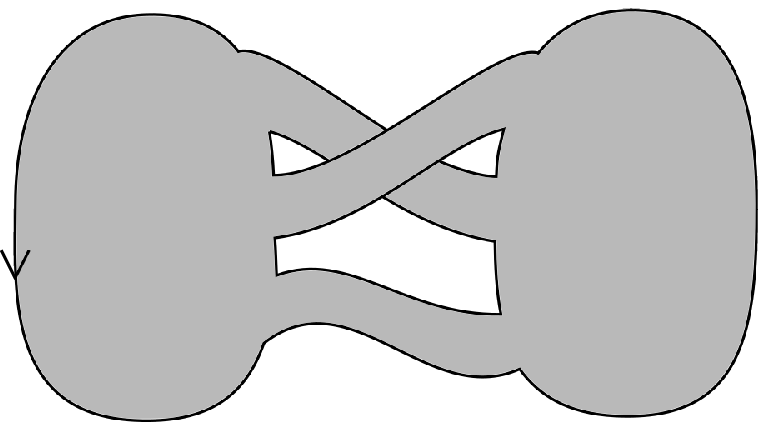\end{array} 
\ar[r]^{\chi} &
\begin{array}{c} \def\svgwidth{2.75in}
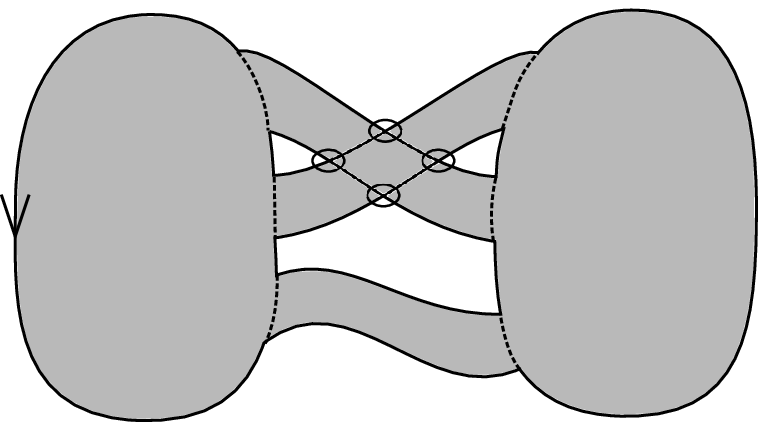\end{array}
} 
$ 
\\ \hline
\end{tabular}
\caption{A surface $\Xi$ (left), a virtual screen $X$ (right), and a projector $\chi:\Xi \to X$.} \label{fig_virt_screen_defn}
\end{figure}


The purpose of the projector and screen is to transfer information from the ambient space of a knot or link $\mathscr{L} \subset \Xi \times I$ to the virtual setting unambiguously. Indeed, suppose $L$ is a link diagram on $\Xi$ and $\chi:\Xi \to X$ is a projector onto $X$. We wish to interpret $\chi(L)$ as a virtual link diagram which $L$ represents. We may assume, after an isotopy if necessary, that $\chi(L)$ is a generically immersed curve in $X \subset \mathbb{R}^2$ and that for all virtual regions $R$ of $X$, no classical crossings of $L$ occur in $\chi^{-1}(R)\approx R' \sqcup R''$. Thus, $L$ intersects each $\chi^{-1}(R)$  in a (possibly empty) union of disjoint arcs. If the image of these arcs intersect in $R$, then they must lie in separate components of $\chi^{-1}(R)$. These intersections can thus be unambiguously interpreted as virtual crossings. Hence each such intersection in $\mathbb{R}^2$ is marked as  a virtual crossing. On the other hand, $\chi$ is one-to-one on a small neighborhood of each classical crossing of $L$, so these intersections may be unambiguously interpreted as classical crossings. We will denote by $\chi(L)$ the resulting virtual link diagram.

\begin{definition}[Virtual Seifert Surface] \label{defn_vss} A \emph{virtual Seifert surface} of an AC knot diagram $\upsilon$ is: a projector $\chi:\Xi \to X$ onto a screen $X$, a knot diagram $K$ on $\Xi$ such that $\chi(K)=\upsilon$, and a Seifert surface $F$ of the corresponding knot $\mathscr{K}$ in $\Xi \times I$. It is denoted by $(F;\chi:\Xi \to X)$. 
\end{definition}  

\begin{theorem}\label{thm_exist} For every AC Gauss diagram $D$, there is an AC virtual knot diagram $\upsilon$ such that $\upsilon$ has Gauss diagram $D$ and $\upsilon$ has a virtual Seifert surface $(F;\chi:\Xi\to X)$.
\end{theorem}
\begin{proof} Let $\Sigma$ be the Carter surface for $D$ and $K$ the knot diagram on $\Sigma$. Note that $D$ is also the Gauss diagram of $K$ on $\Sigma$. Since $K$ is an AC diagram, it is homologically trivial on $\Sigma$. Let $z_0$ be any point interior to some $2$-handle $d_k$ of $\Sigma$. Then $K$ is homologically trivial on $\Xi:=\Sigma\smallsetminus V(z_0)$, where $V(z_0) \subset d_k$ is a small tubular neighborhood of $z_0$. Moreover, the corresponding knot $\mathscr{K}$ in $\Xi \times I$ bounds a Seifert surface $F$ in $\Xi \times I$. Let $\chi':\Xi' \to X$ be the projector onto the screen $X$ shown in Figure \ref{fig_seif_exist}. Here $\Xi'$ is given as a surface in $\mathbb{R}^3\approx \mathbb{R}^2 \times \mathbb{R}$ and $\chi': \Xi' \to X$ is projection to the plane. If $\Sigma$ has genus $g$, then the depicted $\Xi'$ has $2g$ bands attached to the disc. Since $\Xi$ has one boundary component, there is an orientation preserving diffeomorphism from $\Xi$ onto the surface $\Xi'$.  Then $(F;\chi:\Xi\approx \Xi' \to X)$ is a virtual Seifert surface of $\upsilon:=\chi(K)$.
\end{proof}

\begin{figure}[htb]
\begin{tabular}{|c|} \hline \\
$
\xymatrix{
\begin{array}{c} \def\svgwidth{2.75in}
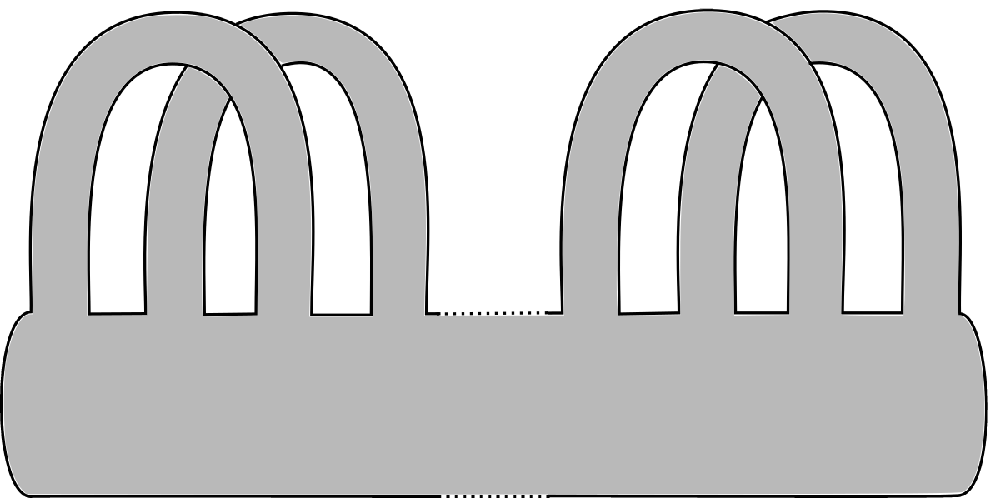\end{array} 
\ar[r]^{\chi'} &
\begin{array}{c} \def\svgwidth{2.75in}
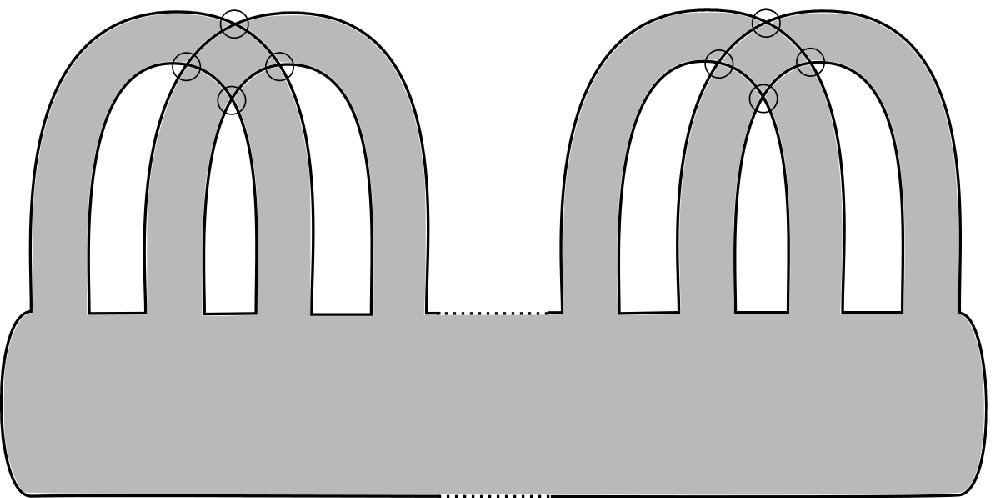\end{array}
} 
$ 
\\ \hline
\end{tabular}
\caption{Screen $X$ and projector $\chi:\Xi \to X$ used in the proof of Theorem \ref{thm_exist}.} \label{fig_seif_exist}
\end{figure}

For a virtual Seifert surface $(F;\chi:\Xi \to X)$ of $\upsilon$, we denote by $\chi(F)$ the image of $F$ under the map $\chi \circ \pi:\Xi \times I \to X$, where $\pi:\Xi \times I \to \Xi$ is canonical projection to the first factor. Since $\chi(\partial F)=\upsilon$, we will say that $\upsilon$ \emph{virtually bounds} $F$. Note that $\chi(F)$ is not necessarily an immersion of $F$. In Section \ref{sec_vss}, it will be shown how to construct a surface $F$ from $D$ and draw $\chi(F)$ in the plane. On the other hand, a virtual Seifert surface can be reconstituted from its image in the case that the image is a virtual band presentation. Furthermore, the following theorem shows that every AC knot has virtual band presentation.

\begin{theorem} \label{thm_vss_band}  For every virtual Seifert surface $(F;\chi:\Xi \to X)$ of $\upsilon$, there is a virtual Seifert surface $(F',\chi:\Xi \to X)$ of $\upsilon'$ such that $\chi(F')=F'_{\tau}$ is a virtual band presentation, $F'$ and $F$ are ambient isotopic in $\Xi \times I$, and $\upsilon \leftrightharpoons \upsilon'$. Conversely, every virtual band presentation $F_{\tau}$ can be realized as a virtual Seifert surface.
\end{theorem}
\begin{proof} For the first claim, note that since $F$ has one boundary component, there is an ambient isotopy in $\Xi \times I$ of $F$ onto a $1$-skeleton of $F$. The $1$-skeleton consists of a canonical system of curves on $F$ that meet in a single point. By a standard argument (see \cite{bz}, Proposition 8.2), this is ambient isotopic in $\Xi \times I$ to a surface consisting of an embedded disc $B_0$ in $\Xi\times I$ with framed $1$-handles $h_1,\ldots, h_n$ attached to $\partial B_0$. The cores of the $h_i$ can be arranged so that their images $a_i$ under the coordinate projection $\Xi \times I \to \Xi$ intersect transversely in double points. If $R$ is a virtual region of $X$, then $\chi^{-1}(R)$ is a disjoint union of two discs in $\Xi$. Hence any intersections of the $a_i$ on $\Xi$ may be assumed to occur in $\Xi \smallsetminus \bigcup_R \chi^{-1}(R)$. Since $F$ is orientable, the framing of each $h_i$ is some integer number of $\pm$-full twists. Thus the framing of each $h_i$ may be indicated by a sequence of $\pm$-curls (see Figure \ref{fig_proj_oto}, top left and right). Applying $\chi$ to $B_0 \cup a_1 \cup \cdots \cup a_n$ gives the desired virtual band presentation. For the second claim, construct $\Xi$ by applying the Carter surface algorithm (see Figure \ref{fig_zeroone_hand_attach}) to the tangle $\tau$ underlying $F_{\tau}$. Details are left to the reader.
\end{proof}

\section{Virtual Seifert Surface Algorithm}
\label{sec_vss}
\subsection{The algorithm in brief} \label{sec_exec} 
The main difference between the virtual Seifert surface algorithm and its classical counterpart is that in the former we draw the subsurfaces bounded by the Seifert cycles first. A virtual knot diagram does not appear until the very end. Subsurfaces are found by writing them as linear combinations of $2$-handles in the Carter surface. The linear combinations $Y_1,\ldots,Y_l$ of $2$-handles are called a \emph{spanning solution}. A virtual screen is constructed in $\mathbb{R}^2$ by gluing together the $2$-handles occurring in the spanning solution. The subsurfaces are then drawn on the virtual screen according to the spanning solution. Half-twisted bands can then be attached at the crossings to complete the virtual knot diagram. For convenience, we outline these steps below.
\newline
\newline
\centerline{
\fbox{\parbox{5.2in}{\underline{\textbf{Virtual Seifert Surface Algorithm:}} Let $D$ be an AC Gauss diagram and $\Sigma$ its Carter surface. Let $C_0,C_1,C_2$ be the standard handle decomposition of $\Sigma$.
\newline
\begin{tabular}{|c|}  \hline \sffamily 1 \\ \hline \end{tabular} Compute $\partial_2:C_2 \to C_1$ and Seifert cycles from $D$. 
\newline
\begin{tabular}{|c|}  \hline \sffamily 2 \\ \hline \end{tabular} Find a spanning solution $Y_1,\ldots,Y_l$ missing some $2$-handle $d_k$
\newline
\begin{tabular}{|c|}  \hline \sffamily 3 \\ \hline \end{tabular} Make the virtual screen $X$ from the $2$-handles in the spanning solution.
\newline
\begin{tabular}{|c|}  \hline \sffamily 4 \\ \hline \end{tabular} Glue together the $2$-handles in each $Y_j$ to make the subsurfaces. 
\newline
\begin{tabular}{|c|}  \hline \sffamily 5 \\ \hline \end{tabular} Attach half-twisted bands at the crossings to the subsurfaces. 
}}
}
\newline
\newline
Three illustrating examples will be used throughout: 4.99, 5.2025, and 6.87548. For the reader's convenience, a summary is given in Figures \ref{fig_4pt99_summ}, \ref{fig_5pt2025}, and \ref{fig_6pt87548}, respectively. Details of the algorithm are explained in the sections below. For Step 1, refer back to Section \ref{sec_homology}.

\begin{figure}[p]
\makebox[\textwidth][c]{
\begin{tabular}{|c|cc|} \hline
\multicolumn{3}{|l|}{\textbf{Step 1A:} Perform Seifert smoothing. Make alternating. Perform all $A$ and all $B$ smoothings. } \\ \hline & & \\
\underline{4.99, Seifert smoothing:} & \underline{All $A$ state:} & \underline{All $B$ state:} \\ & & \\
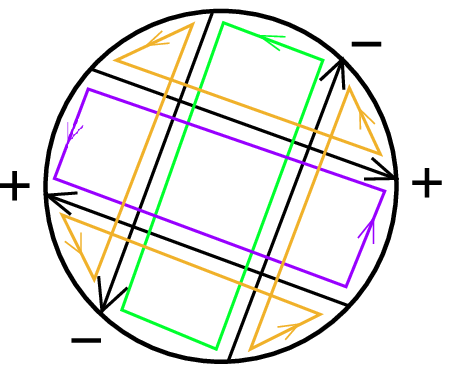 & \input{gd4pt99_2.eps_tex} & \input{gd4pt99_3.eps_tex}\\ & & \\
\begin{tabular}{c}
\underline{\textbf{Step 1B:} Orient $2$-handles.} \\ \\ \\ \def\svgwidth{1.5in} 
\input{discs_orient_4pt99.eps_tex} \\ (marked on diagrams above) \end{tabular} & \multicolumn{2}{|l|}{\begin{tabular}{l} \underline{\textbf{Step 1C:} Compute $\partial_2:C_2 \to C_1$ and Seifert cycles.} \\ \\ $\begin{array}{cll}
\partial_2^{\uptau} &=& \left[\begin{array}{cccccccc}  1  & -1 & 1 & -1  & 1 & -1 & 1    & -1  \\ 
                                                                               -1 & 0  &  0 & 0 &  -1  & 0 &  0  & 0 \\ 
                                                                               0  & 1  & 0 &  1  &  0 & 1  &  0  &  1   \\ 
                                                                               0  & 0  & -1 &  0  &  0 & 0  &  -1  &  0   \end{array} \right] \begin{array}{c} d_1 \\ d_2 \\ d_3 \\ d_4
\end{array}\\ \\
s_1^{\uptau} &=& \left[ \begin{array}{cccccccc} 1 & 0 & 0 & 0 & 1 & 0 & 0 & 0  \end{array} \right]\\ \\
s_2^{\uptau} &=& \left[ \begin{array}{cccccccc} 0 & 1 & 0 & 1 & 0 & 1 & 0 & 1  \end{array} \right]\\ \\
s_3^{\uptau} &=& \left[ \begin{array}{cccccccc} 0 & 0 & 1 & 0 & 0 & 0 & 1 & 0  \end{array} \right]\\ 
\end{array}$
\end{tabular}} \\ \hline \multicolumn{3}{|l|}{\underline{\textbf{Step 2:} Find a spanning solution.}}\\ \multicolumn{3}{|c|} {$s_1=\partial_2(-d_2), \,\,\,\, s_2=\partial_2(d_3),\,\,\, s_3=\partial_2(-d_4)$} \\ \hline
\multicolumn{3}{|l|}{\begin{tabular}{l|r} \underline{\textbf{Steps 3,4,5:} See full details in Figure \ref{fig_4pt99_screen}.} & \underline{Virtual band presentation of 4.99.} \\ & \\ \begin{tabular}{l} \def\svgwidth{3in} 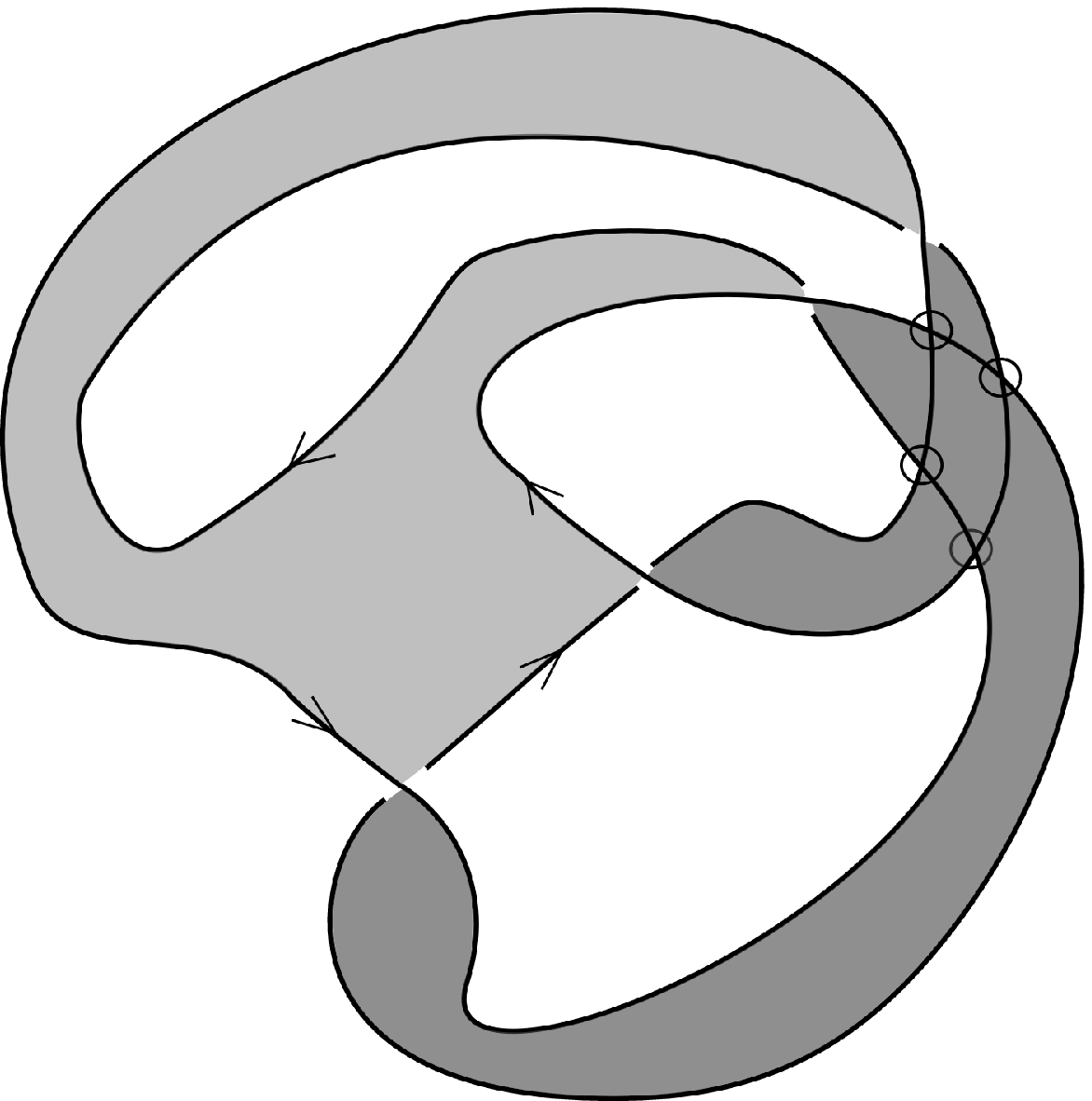 \end{tabular} & \begin{tabular}{c} \def\svgwidth{2.25in} 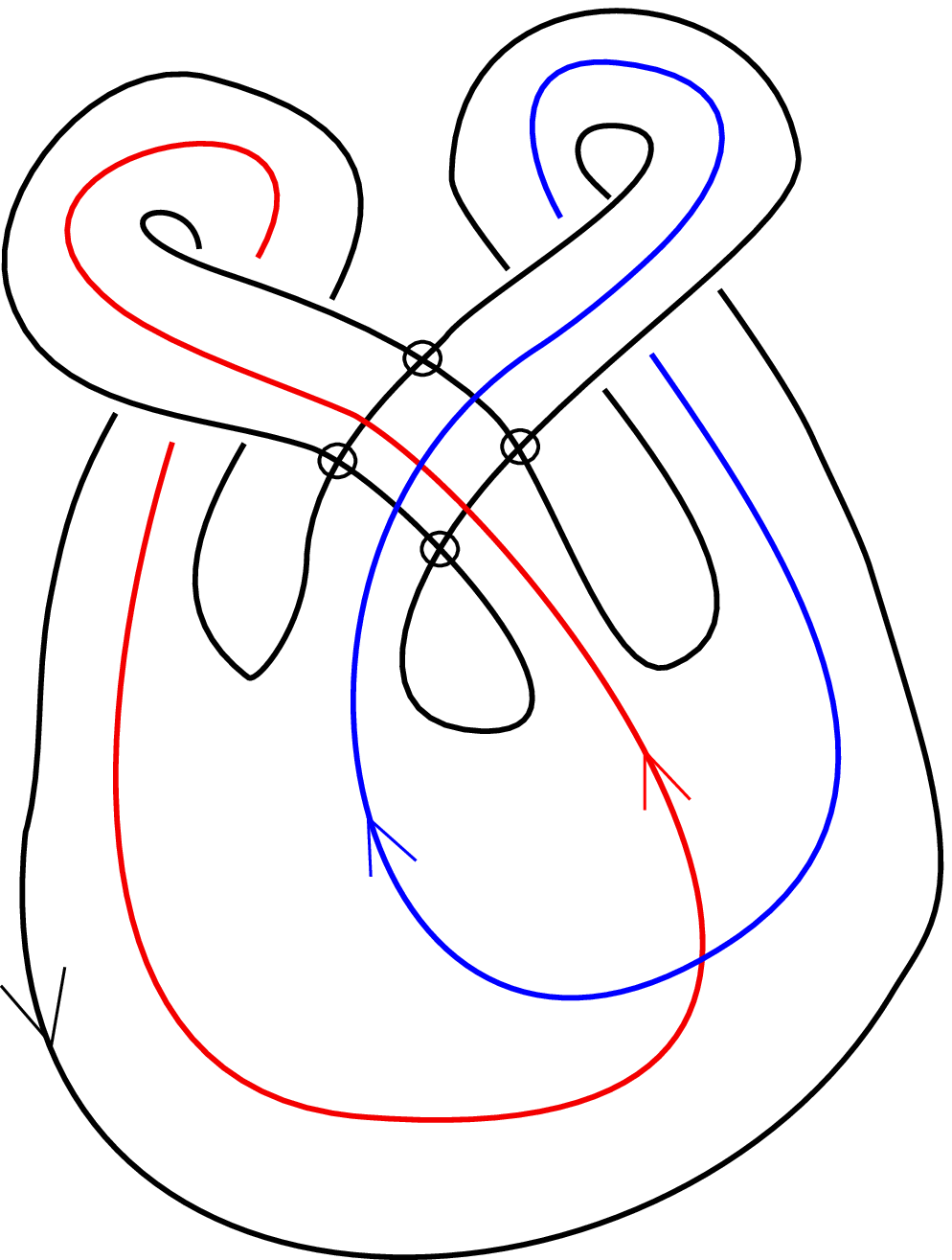\end{tabular} \end{tabular}} \\\hline
\end{tabular}}
\caption{Summary of virtual Seifert surface algorithm for 4.99.} \label{fig_4pt99_summ}
\end{figure}

\begin{figure}[p]
\makebox[\textwidth][c]{
\begin{tabular}{|c|cc|} \hline
\multicolumn{3}{|l|}{\textbf{Step 1A:} Perform Seifert smoothing. Make alternating. Perform all $A$ and all $B$ smoothings. } \\ \hline & & \\
\underline{5.2025, Seifert smoothing:} & \underline{All $A$ state:} & \underline{All $B$ state:} \\ & & \\
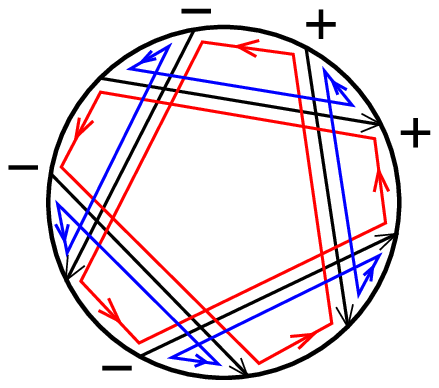 & 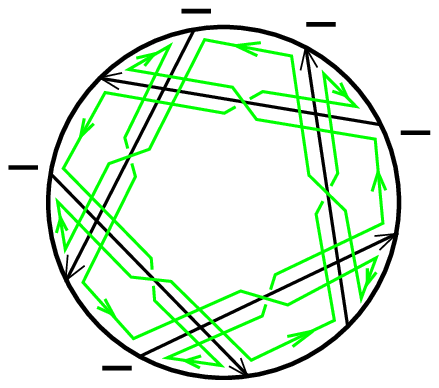 & 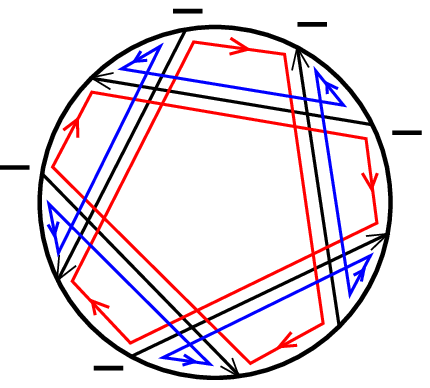\\ & & \\
\begin{tabular}{c}
\underline{\textbf{Step 1B:} Orient $2$-handles.} \\ \\ \\ \def\svgwidth{1.5in} 
\input{discs_orient_test.eps_tex} \\ (marked on diagrams above) \end{tabular} & \multicolumn{2}{|l|}{\begin{tabular}{l} \underline{\textbf{Step 1C:} Compute $\partial_2:C_2 \to C_1$ and Seifert cycles.} \\ \\ $\begin{array}{cll}
\partial_2^{\uptau} &=& \left[\begin{array}{cccccccccc} -1 & 1 & -1 & 1  & -1 & -1 & 1 & -1 & 1 & -1 \\ 
                                                                 0 & -1&  0 & -1 &  0  & -1 & 0 & -1 & 0 & -1 \\ 
                                                                 1 & 0  & 1 &  0  &  1 & 0  & 1  & 0  & 1 & 0 \end{array} \right] \begin{array}{c} d_1 \\ d_2 \\ d_3 \end{array}\\ \\
s_1^{\uptau} &=& \left[ \begin{array}{cccccccccc} 0 & 1 & 0 & 1 & 0 & 1 & 0 & 1 & 0 & 1 \end{array} \right]\\ \\
s_2^{\uptau} &=& \left[ \begin{array}{cccccccccc} 1 & 0 & 1 & 0 & 1 & 0 & 1 & 0 & 1 & 0 \end{array} \right]\\
\end{array}$
\end{tabular}} \\ \hline \multicolumn{3}{|l|}{\underline{\textbf{Step 2:} Find a spanning solution.}}\\ \multicolumn{3}{|c|} {$s_1=\partial_2(-d_2), \,\,\,\, s_2=\partial_2(d_3)$} \\ \hline
\multicolumn{3}{|l|}{\begin{tabular}{l|c} \underline{\textbf{Steps 3,4,5:} See full details in Figure \ref{fig_5pt2025_screen}.} & \underline{ Virtual band presentation (See Fig. \ref{fig_5pt2025_deform}).} \\ & \\ \begin{tabular}{l} 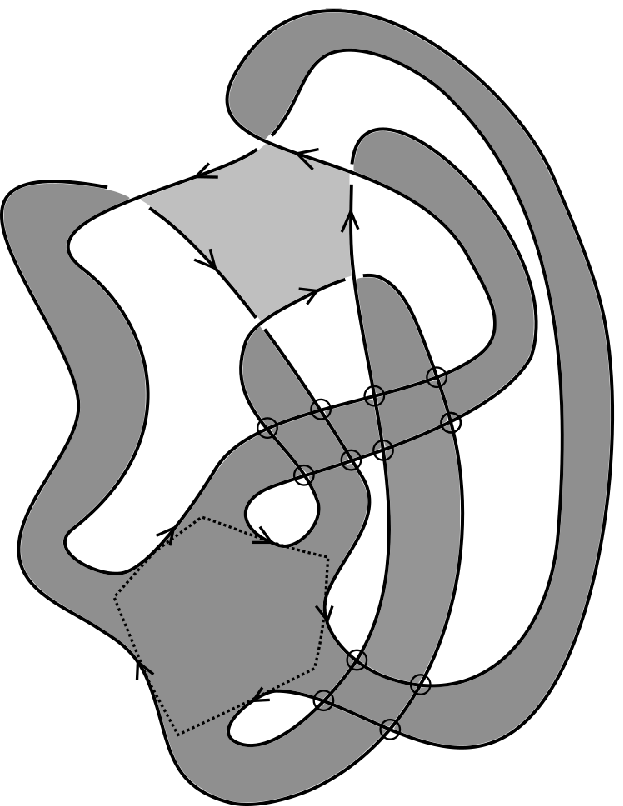 \end{tabular} & \begin{tabular}{r} \def\svgwidth{3in} \input{5pt2025_band.eps_tex}\end{tabular} \end{tabular}} \\\hline
\end{tabular}}
\caption{Summary of virtual Seifert surface algorithm for 5.2025.} \label{fig_5pt2025}
\end{figure}

\begin{figure}[p]
\makebox[\textwidth][c]{
\begin{tabular}{|c|cc|} \hline
\multicolumn{3}{|l|}{\textbf{Step 1A:} Perform Seifert smoothing. Make alternating. Perform all $A$ and all $B$ smoothings. } \\ \hline & & \\
\underline{6.87548, Seifert smoothing:} & \underline{All $A$ state:} & \underline{All $B$ state:} \\ & & \\
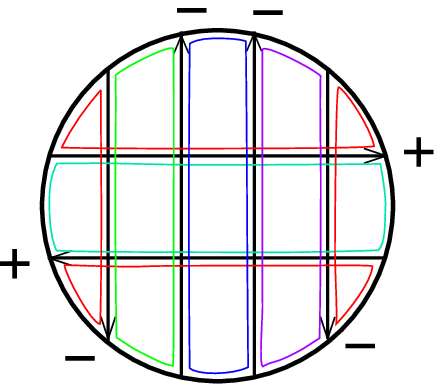 & 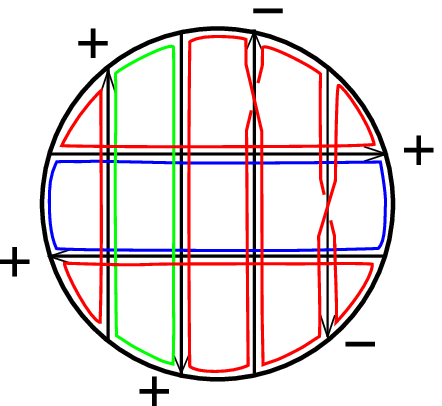 & 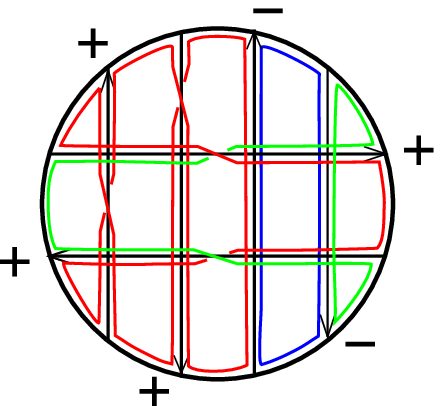\\ & & \\
\begin{tabular}{c}
\underline{\textbf{Step 1B:} Orient $2$-handles.} \\ \\ \\ \def\svgwidth{1.5in} 
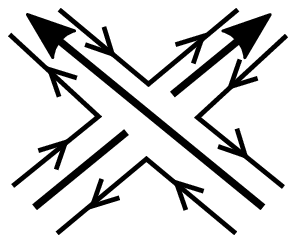 \\ (mark on diagrams, exercise) \end{tabular} & \multicolumn{2}{|l|}{\begin{tabular}{l} \underline{\textbf{Step 1C:} Compute $\partial_2:C_2 \to C_1$ and Seifert cycles.} \\ \\ 
\tiny
$\begin{array}{cll}
\partial_2^{\uptau} &=&  
\left[\begin{array}{cccccccccccc}  
-1 & 1 & 0 & 1 & 0 & 1 & 0 & 1 & -1 & 1 & 0 & 1 \\ 
0 & 0 & -1 & 0 & 0 & 0 & -1 & 0 & 0 & 0 & 0 & 0 \\ 
0 & 0 & 0 & 0 & -1 & 0 & 0 & 0 & 0 & 0 & -1 & 0 \\
0 &-1 & 1 & -1 & 0 & -1 & 1 & -1 & 0 & 0 & 1 & 0 \\
0 & 0 & 0 & 0 & 1 & 0 & 0 & 0 & 0 & -1 & 0 & -1 \\
1 & 0 & 0 & 0 & 0 & 0 & 0 & 0 & 1 & 0 & 0 & 0 \\ 
\end{array}\right] \begin{array}{c} d_1 \\ d_2 \\ d_3 \\ d_4 \\ d_5 \\ d_6 \end{array}\\ \\
s_1^{\uptau} &=& \left[ \begin{array}{cccccccccccc} 1 & 0 & 0 & 0 & 0 & 0 & 0 & 0 & 1 & 0 & 0 & 0 \end{array} \right] \\ \\
s_2^{\uptau} &=& \left[ \begin{array}{cccccccccccc} 0 & 1 & 0 & 0 & 0 & 0 & 0 & 1 & 0 & 0 & 0 & 0 \end{array} \right] \\  \\
s_3^{\uptau} &=& \left[ \begin{array}{cccccccccccc} 0 & 0 & 1 & 0 & 0 & 0 & 1 & 0 & 0 & 0 & 0 & 0 \end{array} \right] \\ \\
s_4^{\uptau} &=& \left[ \begin{array}{cccccccccccc} 0 & 0 & 0 & 0 & 1 & 0 & 0 & 0 & 0 & 0 & 1 & 0 \end{array} \right] \\  \\
s_5^{\uptau} &=& \left[ \begin{array}{cccccccccccc} 0 & 0 & 0 & 1 & 0 & 1 & 0 & 0 & 0 & 1 & 0 & 1 \end{array} \right] \\ 
\end{array}$
\normalsize
\end{tabular}} \\ \hline \multicolumn{3}{|l|}{\underline{\textbf{Step 2:} Find a spanning solution.}}\\ \multicolumn{3}{|c|}{} \\ \multicolumn{3}{|c|} {$s_1=\partial_2(d_6),\,\, s_3=\partial_2(-d_2), \,\, s_4=\partial_2(-d_3), \,\, s_2+s_5=\partial_2(-d_2-d_3-d_4-d_5)$.} \\ \multicolumn{3}{|c|}{} \\ \hline \multicolumn{3}{|l|}{\begin{tabular}{l|c} \underline{\textbf{Steps 3,4,5:} See full details in Figure \ref{fig_6pt87548_screen}.} & \underline{Virtual band presentation of 6.87548.} \\ &\\ \begin{tabular}{l} 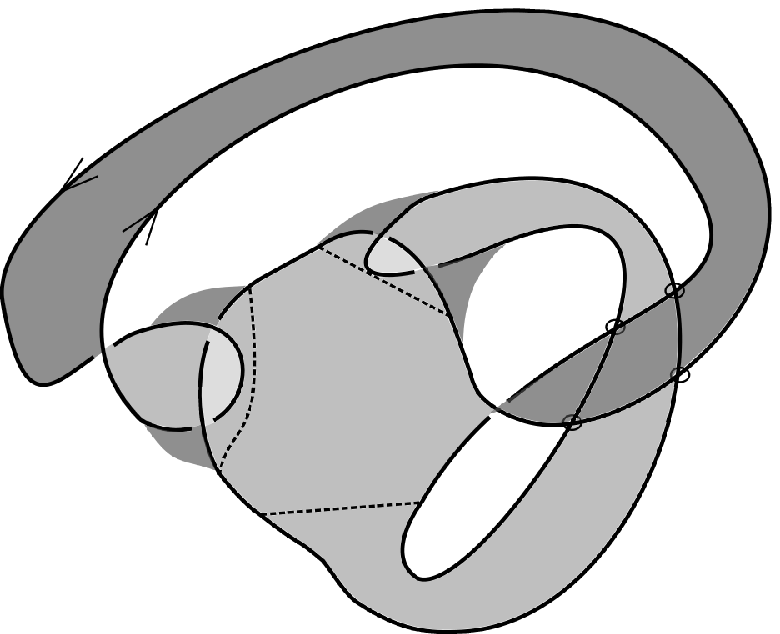 \end{tabular} & \begin{tabular}{l} \\ \def\svgwidth{3in} 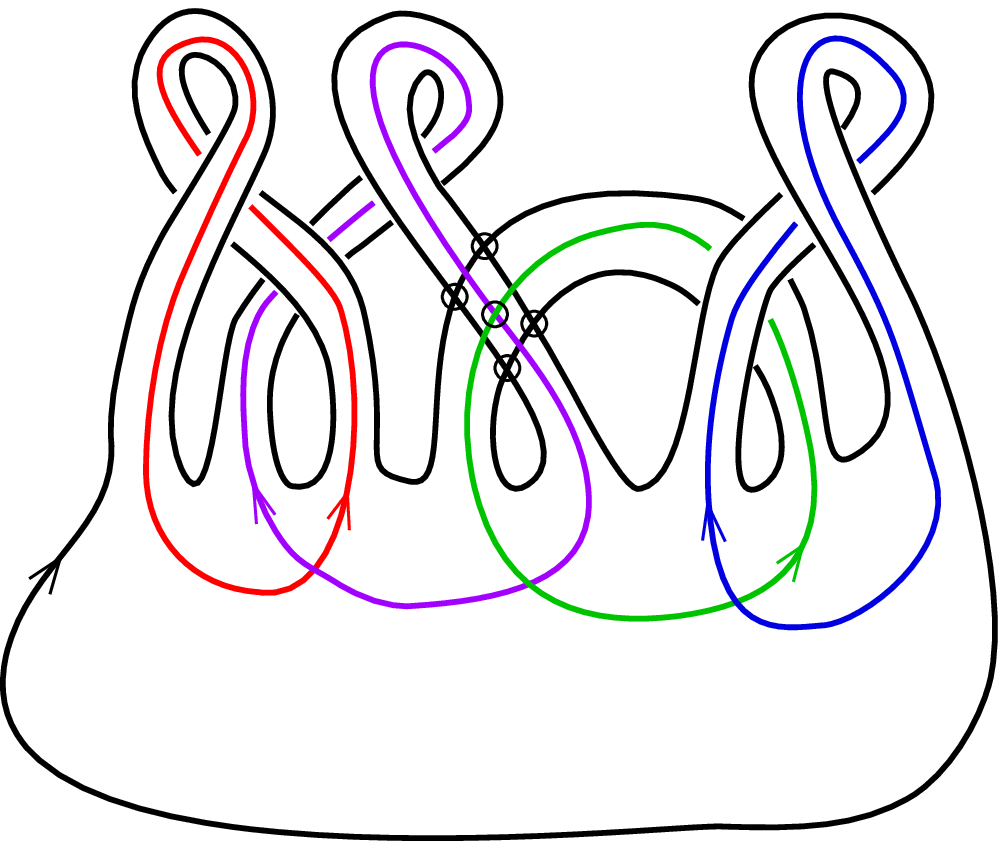\end{tabular} \end{tabular}} \\\hline
\end{tabular}}
\caption{Summary of virtual Seifert surface algorithm for 6.87548.}\label{fig_6pt87548}
\end{figure}

\subsection{Step 2: Spanning solutions} We begin with an algebraic reinterpretation of the constructive proof for the existence of Seifert surfaces for AC knots that was sketched in Section \ref{sec_AC}. Let $K$ be an AC diagram on its Carter surface $\Sigma$. Performing the Seifert smoothing gives the set of Seifert cycles $G=\{s_1,\ldots,s_p\}$. Since $K$ is AC, the sum of these pairwise disjoint simple curves on $\Sigma$ is homologically trivial in $H_1(\Sigma)$. The set of Seifert cycles $G$ can be partitioned into non-empty subsets $G_1, \cdots, G_l$ such that $\sum_{s_i \in G_j} s_i \in \text{im}(\partial_2)$. The $G_j$ may be chosen to be \emph{minimal} in the sense that no subset of $G_j$ also satisfies this property. For each $j$, there are integers $y_{j,1},\ldots, y_{j,m}$ such that:
\begin{equation} \label{eqn_span}
\partial_2(y_{j,1} d_1+ y_{j,2} d_2+ \ldots+y_{j,m} d_m)=\sum_{s_i \in G_j} s_i.
\end{equation}
Here we may assume that each $y_{j,k}$ is either $0$ or $\pm 1$. For $1 \le j \le l$, set $S_j$ to be the subsurface of $\Sigma$ consisting of the union of all the $2$-handles $d_k$ such that $y_{j,k} \ne 0$, together with any $0$- and $1$-handles incident to them. Since $S_j$ is oriented, we may assume that either $y_{j,k} \in \{0,1\}$ for all $k$ or $y_{j,k} \in \{-1,0\}$ for all $k$.  The common sign of all $y_{j,k} \ne 0$ determines the orientation of $S_j$, where $+$ means that $S_j$ carries the orientation of $\Sigma$ and $-$ means it carries the opposite orientation. The minimality of $G_j$ ensures that each $S_j$ is connected. A Seifert surface for the corresponding knot $\mathscr{K} \subset \Sigma \times I$ is made from the subsurfaces $S_1,\ldots,S_l$ as in Section \ref{sec_AC}. Thus, the most important part of this construction is the set of solutions to Equation \ref{eqn_span}. This motivates the following definition.

\begin{definition} Let $D$ be an AC Gauss diagram, let $G=\{s_1,\ldots,s_p\}$ be the set of Seifert cycles, and $d_1,\ldots,d_m$ the $2$-handles. A \emph{spanning solution} is a minimal partition $G_1,\ldots,G_l$ of $G$ described above together with a solution $Y_j$ to Equation \ref{eqn_span} for each $j=1,\ldots,l$. Each $Y_j$ is called a \emph{spanning subsolution}. A spanning solution is said to be \emph{missing} a $2$-handle $d_k$ if $y_{j,k}=0$ for $1 \le j \le l$.
\end{definition}

\begin{theorem} \label{lemma_miss} For all $2$-handles $d_k$ of the Carter surface $\Sigma$, there is a spanning solution missing $d_k$. For every spanning solution missing $d_k$, there is a virtual Seifert surface $(F,\chi:\Xi \to X)$, with $\Xi=\overline{\Sigma \smallsetminus d_k}$.  
\end{theorem}
\begin{proof} Let $(\{G_1,\ldots,G_l\},\{Y_1,\ldots,Y_l\})$ be any spanning solution and let $d_k$ be a given $2$-handle. Suppose there is a subsolution $Y_j$ where $d_k$ occurs with coefficient $y_{j,k} \ne 0$. Then every $2$-handle $d_i$ with $y_{j,i} \ne 0$ satisfies $y_{j,i}=y_{j,k}$. Since $\partial_2(d_1+\ldots+d_m)=0$, we have that:
\[
\partial_2(y_{j,1} d_1+\ldots+y_{j,m} d_m-y_{jk}(d_1+\ldots+d_m))=\sum_{s_i \in G_j} s_i. 
\]
Thus, we have a new spanning subsolution missing $d_k$. By appropriately adding or subtracting a copy $d_1+\ldots+d_m$ of the Carter surface to each subsolution $Y_j$ with $y_{j,k} \ne 0$, we obtain a spanning solution missing $d_k$. The second claim follows from Theorem \ref{thm_exist} and the preceding remarks.
\end{proof}

\begin{remark} Spanning solutions for a given Gauss diagram are not unique.
\end{remark}

\textbf{Example (4.99):} The map $\partial_2: C_2 \to C_1$ was computed in Section \ref{sec_homology}. From Figure \ref{fig_4pt99_summ}, it follows that the Seifert smoothing gives three cycles $s_1,s_2,s_3$, where $s_1=\partial_2 (-d_2)$, $s_2=\partial_2 d_3$ and $s_3=\partial_2(-d_4)$. Setting $G=\{s_1,s_2,s_3\}$, this gives a minimal partition of $G_1=\{s_1\}$, $G_2=\{s_2\}$, $G_3=\{s_3\}$. This is a spanning solution missing $d_1$ where the subsurfaces $S_1,S_2,S_3$ are all discs.
\newline
\newline
\textbf{Example (5.2025):} The map $\partial_2:C_2 \to C_1$ is given in Figure \ref{fig_5pt2025}. There are two Seifert cycles, $s_1,s_2$. Observe that that $s_1=\partial_2(-d_2)$ and $s_2=\partial_2(d_3)$. This gives a minimal partition of $G=\{s_1,s_2\}$ into $G_1=\{s_1\}$ and $G_2=\{s_2\}$. This is a spanning solution missing $d_1$. There are two subsurfaces $S_1$, $S_2$, both of which are discs.
\newline
\newline
\textbf{Example (6.87548):} From the computation of $\partial_2:C_2 \to C_1$ in Figure \ref{fig_6pt87548}, we find a spanning solution for $D$. One possibility is: $s_1=\partial_2(d_6)$, $s_3=\partial_2(-d_2)$, $s_4=\partial_2(-d_3)$, and $s_2+s_5=\partial_2(-d_2-d_3-d_4-d_5)$. It can be checked that this gives a minimal partition of $G=\{s_1,s_2,s_3,s_4,s_5\}$ as $G_1=\{s_1\}$, $G_2=\{s_3\}$, $G_3=\{s_4\}$, $G_4=\{s_2,s_5\}$. Thus there are a total of four subsurfaces $S_1,S_2,S_3,S_4$, where each is given as a sum of oriented discs. The spanning solution in this case misses $d_1$.

\subsection{Steps 3,4,5: Construction of virtual Seifert surfaces} Given a spanning solution missing some $2$-handle, a virtual Seifert surface is constructed in three steps: (1) forming the virtual screen $X$ and projector $\chi: \Xi \to X$, (2) placing the subsurfaces $S_1,\ldots,S_l$ on $X$, and (3) attaching the subsurfaces by half-twisted bands at the crossings of $K$. 
\newline
\newline
First we construct $\Xi$ from a spanning solution that misses some $2$-handle, say $d_m$. Let $d_1,\ldots,d_m$ be the $2$-handles of the Carter surface $\Sigma$ of $D$ and $(\{G_1,\ldots,G_l\},\{Y_1,\ldots,Y_l\})$ a spanning solution of $D$ missing $d_m$.  Let $Z=\bigcup_{j=1}^l \{d_k: y_{j,k} \ne 0\}$ and $z=|Z|$. Since the spanning solution misses $d_m$, $z<m$. Let $\Xi$ be the subsurface of $\Sigma$ consisting of all the $0$- and $1$-handles of $\Sigma$, together with all the $2$-handles in $Z$. Note that since $z<m$, $\Xi$ has at least one boundary component.  
\newline
\newline
The virtual screen $X$ is built from embeddings of the $d_i\in Z$ into $\mathbb{R}^2$. In particular, each $2$-handle $d_i \in Z$ may be embedded as a counterclockwise oriented polygon $P_i$ in $\mathbb{R}^2$. The sides of $P_i$ are in one-to-one correspondence with the $1$-handles in $\partial_2 d_i$. Moving in the counterclockwise direction on $\partial P_i$, label the sides of $\partial P_i$ with the $1$-handles in $\partial_2 d_i$ so that their ordering is preserved while traversing $\partial_2 d_i$ in its direction of orientation. A side labeled $c_j$ of $\partial P_i$ is directed counterclockwise if it occurs with coefficient $+1$ in $\partial_2 d_i$ and clockwise if it occurs with coefficient $-1$. Similarly, the vertices of each $P_i$ may be labeled with the corresponding $0$-handles of $\Sigma$ (i.e. crossings of $K$). Typically, we suppress that labels on the $0$-handles in figures as they can also be distinguished by the $1$-handle labels. Now, if $Z=\{d_{i_1},\ldots,d_{i_z}\}$, embed the polygons $P_{i_1},\ldots, P_{i_z}$ disjointly into $\mathbb{R}^2$.
\newline
\newline
To finish the virtual screen, add in all the $1$-handles and $0$-handles of $\Sigma$. If $P_i$ and $P_j$ have sides with the same label $c_k$, then the $1$-handle $c_k$ is a band $I \times I$ going from $P_i$ to $P_j$. Similarly, if $P_i$ and $P_j$ have corners that share a $0$-handle label, then a band $I\times I$ is attached between a vertex of $P_i$ and a vertex of $P_j$. All these bands may be chosen so that they intersect only in virtual crossings of bands. Note that when a $1$-handle is identified, it is necessary to extend the identification to the $0$-handles incident to it. The final result is a collection of disjoint counterclockwise oriented discs embedded in $\mathbb{R}^2$, attached together by bands that intersect only in virtual crossings of bands. This is a virtual screen $X$. A projector $\chi:\Xi \to X$ is defined by $\chi(d_i)=P_i$ for all $d_i \in Z$ and likewise for the $0$- and $1$-handles.
\newline
\newline
The subsurfaces of a Seifert surface can now be identified in the screen $X$. Indeed, if $S_j$ is the subsurface of $\Xi$ consisting of the union of $2$-handles $d_k$ such that $y_{j,k} \ne 0$, then $\chi(S_j)$ is the union of polygons $P_k$ such that $y_{j,k} \ne 0$, together with any incident lower dimensional handles. The immersed surface $\chi(S_j)$ is oriented counterclockwise in $\mathbb{R}^2$ if the common sign of the nonzero $y_{j,k}$ is $1$ and clockwise if the common sign is $-1$. A Seifert surface $F$ in $\Xi \times I$ is then completed by placing the $S_j$ at different heights and attaching half-twisted bands at the crossings that respect the local writhe and over/under information for the Gauss diagram $D$. Transferring these half-twisted bands to the screen $X$ completes the picture. The virtual knot diagram will now be visible in $X$.
\newline

\begin{figure}[hbt]
\[
\begin{array}{|c|} \hline \\
\xymatrix{\begin{array}{c} \def\svgwidth{1.75in} 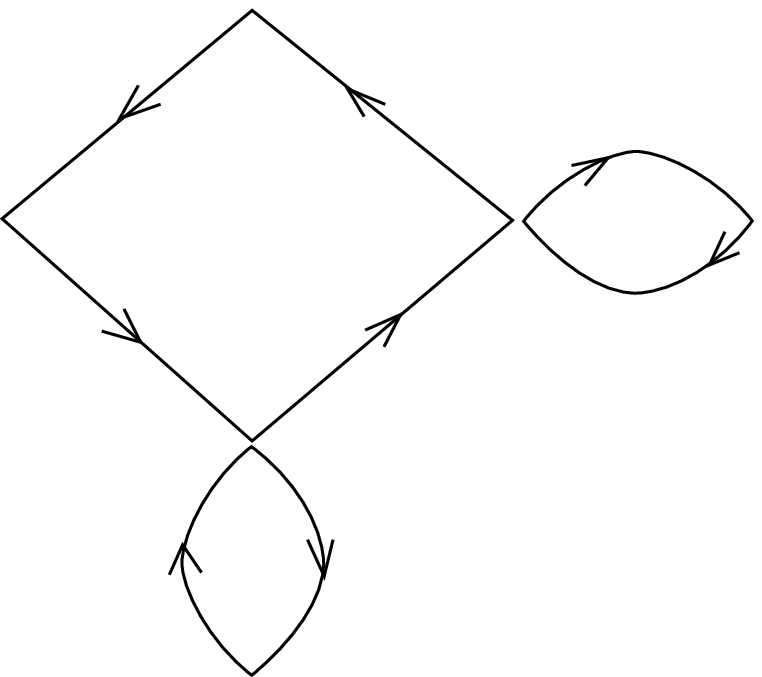 \end{array} \ar[r] & \begin{array}{c} \def\svgwidth{2.25in} 
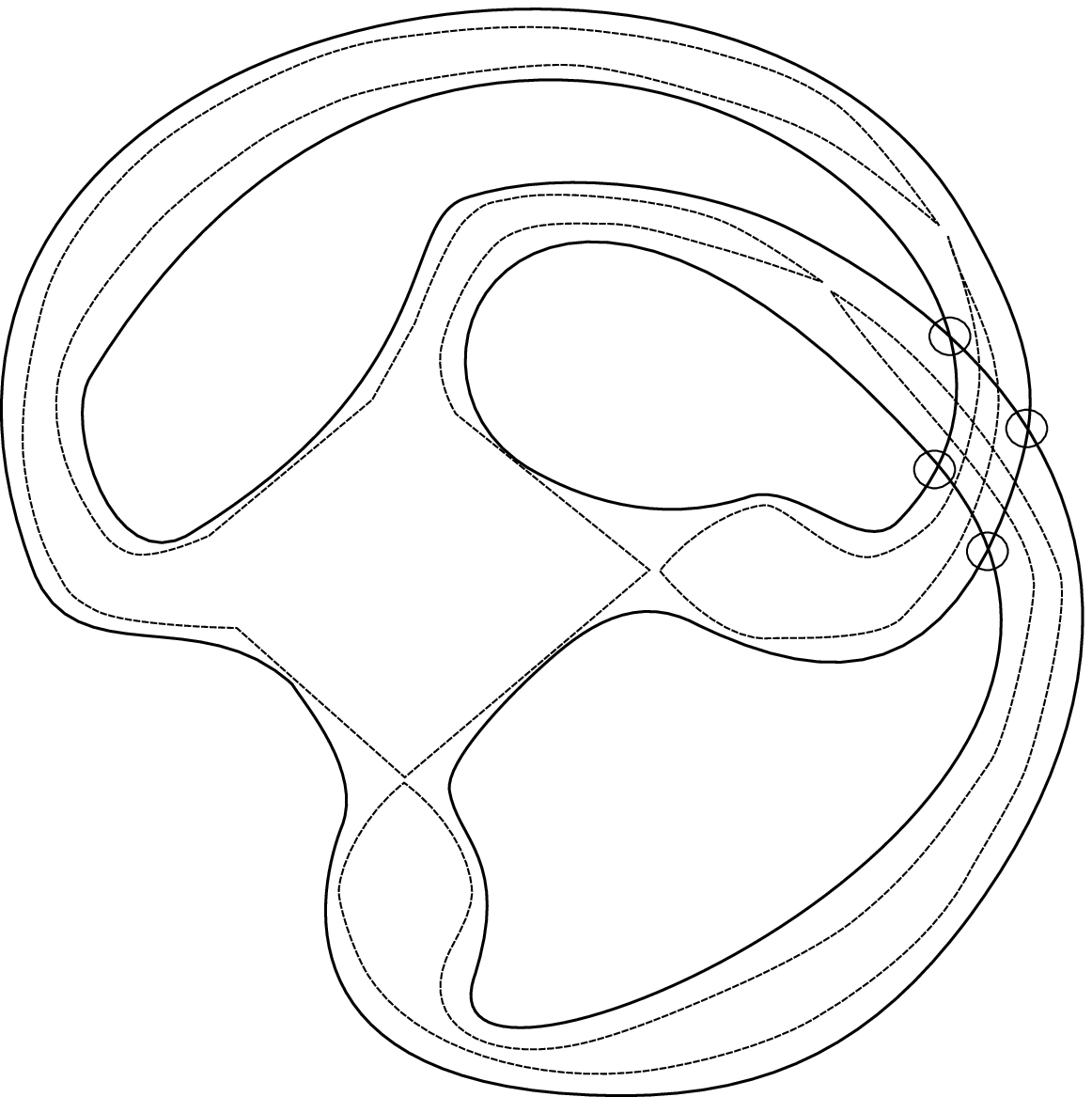 \end{array}} \\
\begin{array}{cc}
\begin{array}{l} \underline{\text{Spanning solution:}} \\ \\ s_1=\partial_2(-d_2) \\ \\ s_2=\partial_2(d_3) \\ \\ s_3=\partial_2(-d_4) \\ \\ (\text{misses } d_1) \end{array} & \begin{array}{|c|} \hline \\ \def\svgwidth{2.4in} 
\input{4pt99_seif.eps_tex} \\ \hline \end{array} \end{array} \\ \\ \hline \end{array}  
\]
\caption{Construction of a virtual screen and a virtual Seifert surface from a spanning solution of 4.99. See also Figure \ref{fig_4pt99_summ}. }\label{fig_4pt99_screen}
\end{figure}

\begin{figure}[hbt]
\[
\begin{array}{|c|} \hline \\
\xymatrix{\begin{array}{c} \def\svgwidth{1in} 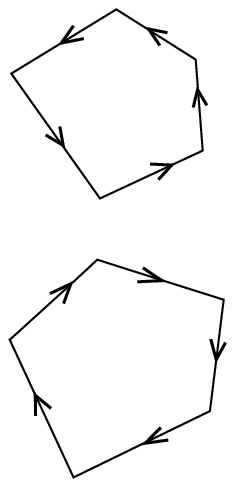 \end{array} \ar[r] & \begin{array}{c} \def\svgwidth{1.75in} 
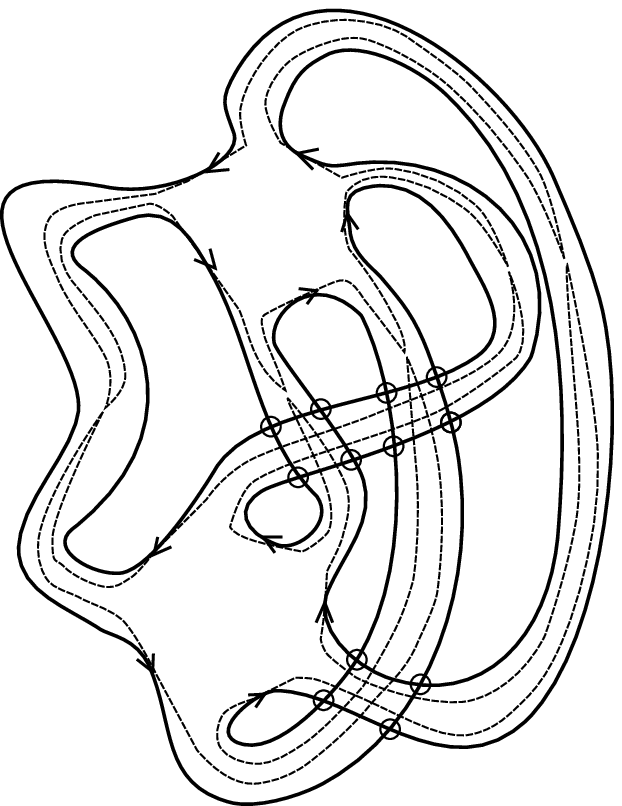 \end{array}} \\
\begin{array}{cc}
\begin{array}{l} \underline{\text{Spanning solution:}} \\ \\ s_1=\partial_2(-d_2) \\ \\ s_2=\partial_2(d_3) \\ \\ (\text{misses } d_1) \end{array} & \begin{array}{|c|} \hline \\ \def\svgwidth{2in} 
\input{5pt2025_surf.eps_tex} \\ \hline \end{array} \end{array} \\ \\ \hline \end{array}  
\]
\caption{Construction of a virtual screen and a virtual Seifert surface from a spanning solution of 5.2025. See also Figure \ref{fig_5pt2025}.}\label{fig_5pt2025_screen}
\end{figure}

\begin{figure}[hbt]
\[
\begin{array}{|c|} \hline \\
\xymatrix{\begin{array}{c} \def\svgwidth{2.5in}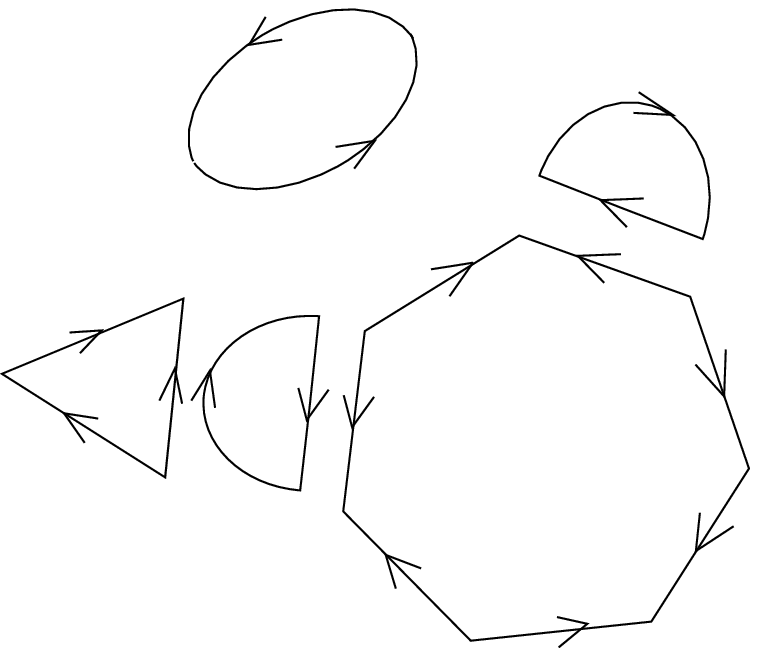 \end{array} \ar[r] & \begin{array}{c} \def\svgwidth{2.5in}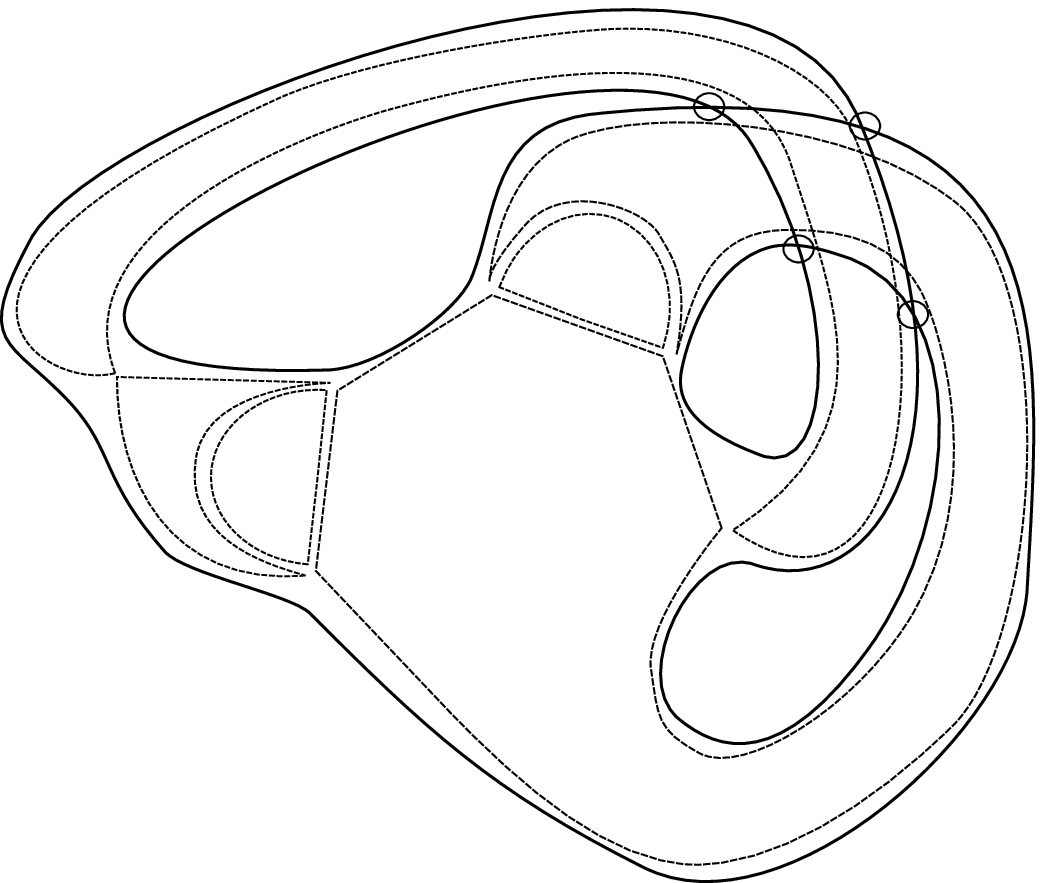 \end{array}} \\
\begin{array}{cc}
\begin{array}{l} \underline{\text{Spanning solution:}} \\ \\ s_1=\partial_2(d_6) \\ \\ s_3=\partial_2(-d_2) \\ \\ s_2+s_5=\partial_2(-d_2-d_3-d_4-d_5) \\ \\ (\text{misses } d_1) \end{array} & \begin{array}{|c|} \hline \\ 
\input{ac_seif_example_2_II.eps_tex} \\ \hline \end{array} \end{array} \\ \\ \hline \end{array}  
\]
\caption{Construction of a virtual screen and a virtual Seifert surface from a spanning solution of 6.87548. See also Figure \ref{fig_6pt87548}.}\label{fig_6pt87548_screen}
\end{figure}

\textbf{Example (4.99):} In Figure \ref{fig_4pt99_screen}, a screen is constructed from the spanning solution previously found. The $2$-handles used in the spanning solution are $d_2$, $d_3$, and $d_4$. Furthermore, $\partial_2(d_2)=c_1+c_5$, $\partial_2(d_3)=c_2+c_8+c_6+c_4$, and $\partial_2(d_4)=c_3+c_7$. Figure \ref{fig_4pt99_screen}, top left, shows $d_2$ and $d_4$ embedded as bigons and $d_3$ embedded as a quadrilateral. There are no polygons sharing $1$-handle labels. Attaching bands to the corners with the same $0$-handle labels gives the virtual screen on the top right in Figure \ref{fig_4pt99_screen}. The dotted arcs show the $2$-handle boundaries in $X$. To draw $F$ in $X$, attach half-twisted bands at the crossings to the subsurfaces $S_1=-d_2$, $S_2=d_3$, and $S_3=-d_4$.
\newline
\newline
\textbf{Example (5.2025):} Refer to the spanning solution from Figure \ref{fig_5pt2025}. The $2$-handles occurring with non-zero coefficient are $d_2$ and $d_3$. Note that $\partial_2 d_3=c_1+c_9+c_7+c_5+c_3$ and $\partial_2 d_2=c_2+c_8+c_4+c_{10}+c_6$. Thus we embed $d_2,d_3$  in $\mathbb{R}^2$ as disjoint pentagons $P_2$, $P_3$, respectively (see Figure \ref{fig_5pt2025_screen}). None of the sides of the polygons have the same $1$-handle label, so it is only necessary to connect the vertices of polygons having the same $0$-handle labels. To draw $F$ in $X$, attach half-twisted bands at the crossings to the oriented subsurfaces $S_1=-d_2, S_2=d_3$.
\newline
\newline
\textbf{Example (6.87548):} Now refer to the spanning solution from Figure \ref{fig_6pt87548}. The $2$-handle $d_1$ is not used, $d_4$, $d_5$, $d_6$ are used once each, and $d_2$, $d_3$ are used twice each. From the Gauss diagram, it is observed that $d_2$, $d_3$, $d_6$ are bigons, $d_5$ is a triangle, and $d_4$ is a heptagon. Figure \ref{fig_6pt87548_screen} shows these embedded disjointly in $\mathbb{R}^2$ and labeled as described above. Attach $1$-handles and $0$-handles as bands, inserting virtual crossings of bands as necessary. There are four subsurfaces $S_1=d_6$, $S_2=-d_2$, $S_3=-d_3$, $S_4=-d_2-d_3-d_4-d_5$. The subsurface $S_4$ is obtained by identifying the corresponding sides of $d_5$, $d_3$, $d_4$, and $d_2$. The subsurfaces $S_2,S_3$ are drawn on top of $d_2,d_3$ respectively. Attach by half-twisted bands at the crossings to obtain the image of $F$ in $X$.

\section{Deforming virtual Seifert surfaces in $\mathbb{R}^2$} \label{sec_step_6}   In this section, we will show how to manipulate virtual Seifert surfaces $(F;\chi:\Xi \to X)$ in the plane. Two techniques are given. The first technique is to project an ambient isotopy of $F$ in $\Xi \times I$ to the virtual screen. It is similar to the method of deforming classical Seifert surfaces (e.g. using \emph{topological script} \cite{on_knots}). The second technique is to alter the virtual screen $X$ and the projector $\chi: \Xi \to X$.  We discuss the details of each of these methods before proceeding to an example. Beginning with the virtual Seifert surface of 5.2025 previously found,  we will use the two methods to find a virtual band presentation of 5.2025. Virtual band presentations of 4.99 and 6.87548 are given in Figures \ref{fig_4pt99_summ} and \ref{fig_6pt87548}, respectively. We will leave the details as exercises for the reader.  

\subsection{Method 1: Projecting an Isotopy} An ambient isotopy $H:(\Xi\times I) \times I \to (\Xi \times I)$ of a Seifert surface $F$ can be drawn on the virtual screen via the map $\chi \circ \pi \circ H$, where $\pi: \Xi \times I \to \Xi$ is projection onto the first factor. A useful special case of this is when $H$ is the identity on the inverse image of the virtual regions. More exactly, suppose that there is an open neighborhood $W$ of $\bigcup_{R} \chi^{-1}(R)$ such that the restriction of $H$ to $(W \times I) \times I$ is the identity, where the union is taken over all virtual regions $R$ of $X$. Observe that the restriction of $\chi$ to $\Xi\smallsetminus W$ is an embedding into $\mathbb{R}^2$. Thus the map $(\chi \times \text{id})$ restricted to $(\Xi \smallsetminus W) \times I$ can be used to trace $H$ in $\mathbb{R}^2 \times I$. It follows that for such $H$, the Seifert surface $F$ may be manipulated in a fashion identical to that of classical Seifert surfaces. Some commonly used local moves are depicted in Figure \ref{fig_proj_oto}. All these moves occur in a small ball in the screen $X$ over which $\chi$ is one-to-one.

\begin{figure}[htb]
\begin{tabular}{cccccc}
\begin{tabular}{c}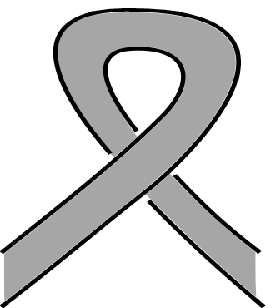 \end{tabular} & $\approx$ & \begin{tabular}{c}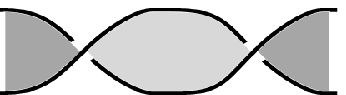\end{tabular}  & 
\begin{tabular}{c}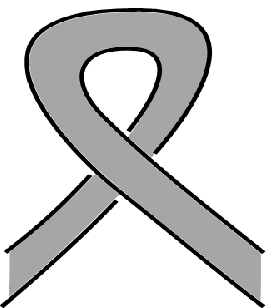 \end{tabular} & $\approx$ & \begin{tabular}{c} 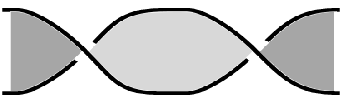 \end{tabular} \\ \\
\begin{tabular}{c}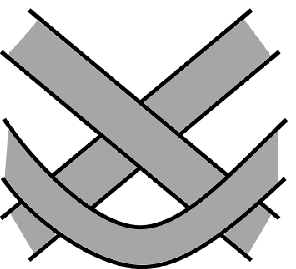 \end{tabular} & $\approx$ & \begin{tabular}{c}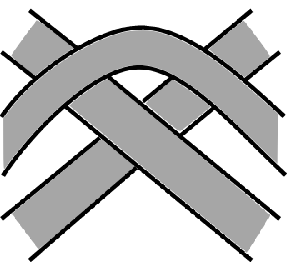\end{tabular}  & 
\begin{tabular}{c}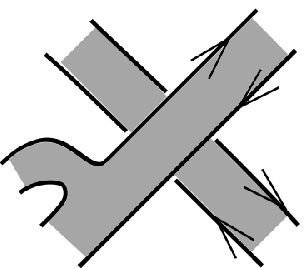 \end{tabular} & $\approx$ & \begin{tabular}{c} 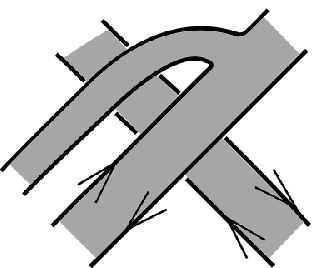 \end{tabular} \\ \\
\begin{tabular}{c}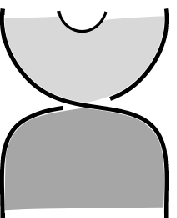 \end{tabular} & $\approx$ & \begin{tabular}{c}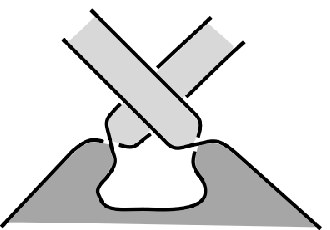\end{tabular}  & 
\begin{tabular}{c}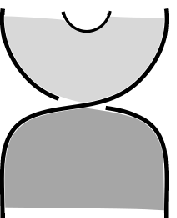 \end{tabular} & $\approx$ & \begin{tabular}{c} 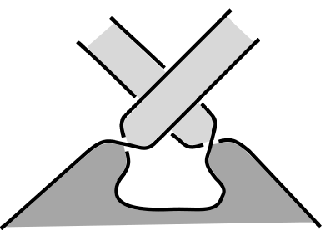 \end{tabular} \\    
\end{tabular}
\caption{Some moves on virtual Seifert surfaces in a neighborhood over which the projector is one-to-one.} \label{fig_proj_oto}
\end{figure}

\subsection{Method 2: Altering the screen and projector} \label{sec_method2} The second way to manipulate a virtual Seifert surface $(F;\chi:\Xi \to X)$ is to keep $F$ and $\Xi$ fixed while altering the virtual screen $X$ and the projector $\chi$.  One example of this would be to change the handle decomposition of $\Xi$. In general, if $(F;\chi:\Xi \to X)$ is given and $\chi':\Xi\to X'$ is another projector and screen, then $(F;\chi':\Xi \to X)$ is also a virtual Seifert surface. If $K$ is a knot diagram on $\Xi$ representing $\partial F$, then $\chi(K)$ and $\chi'(K)$ are equivalent virtual knots. This follows from the fact that they have identical Gauss diagrams. Thus $(F;\chi:\Xi \to X)$ and $(F;\chi':\Xi\to X')$ are virtual Seifert surfaces of equivalent virtual knots. 
\newline
\newline
Figure \ref{fig_proj_alt} gives some local moves on virtual Seifert surfaces obtained by such alterations of the projector and screen. Some care must be taken in the interpretation of these figures. Each move is of the form $LHS \leftrightarrow RHS$. On each side of $\leftrightarrow$, some bands of a virtual screen are drawn. The virtual screen on $LHS$ is \emph{different} than the virtual screen on $RHS$. Moreover, $LHS$ and $RHS$ correspond to \emph{different} projectors of the \emph{same} surface $\Xi$. The moves do not represent an isotopy, but only a change in how $\Xi$ is immersed in $\mathbb{R}^2$. Changing the projector also changes how the image of $F \subset \Xi \times I$ appears in $\mathbb{R}^2$. Figure \ref{fig_proj_alt} indicates the effect of changing the projector on a subsurface of $F$ in several cases. For example, in move (A), the left hand side shows an immersed band of the the virtual screen as a curl with one virtual region. This band can be alternatively embedded in $\mathbb{R}^2$ without the curl, as indicated on the right hand side of the move. In addition, if the depicted gray surface indicates the image of a piece of $F$ in $\mathbb{R}^2$, then the effect on the image of $F$ due to the change of $\chi$ is also to uncurl the virtual curl. 
\newline
\newline
Moves (B) and (C) of Figure \ref{fig_proj_alt} may be similarly interpreted. Changing the projector and the way the bands of the virtual screen intersect has a corresponding effect on the image of $F$ in $\mathbb{R}^2$. For move (D), the virtual screen is changed by altering where the two bands intersect in a virtual crossing of bands.  The effect on the image of $F$ is that the half-twisted band at the crossing gets moved to the other side of the virtual crossings of bands. It is interesting to observe that the effect of the move on the virtual knot is two extended Reidemeister moves (see Figure \ref{fig_reid_moves}). 
\newline
\newline
In order to reduce complexity of the figures, we generally do not draw the virtual screen while the virtual Seifert surface is being deformed. However, it is important to emphasize that any manipulation of the virtual Seifert surface must come from either an ambient isotopy of $F$ in $\Xi \times I$ or by an alteration of the the projector and virtual screen. Thus it is essential to keep track of where the virtual screen is at all times, even if it is not explicitly drawn on the paper. This is not difficult in practice, as is seen in the following example.
\newline

\begin{figure}[htb]
\begin{tabular}{cccccc}
\begin{tabular}{c}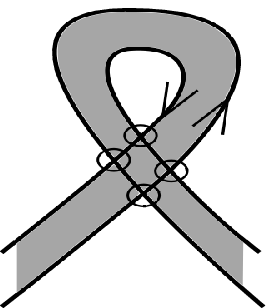 \end{tabular} & $\leftrightarrow$ & \begin{tabular}{c}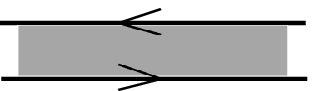\end{tabular}  & 
\begin{tabular}{c}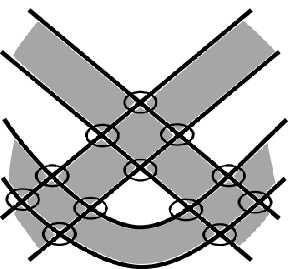 \end{tabular} & $\leftrightarrow$ & \begin{tabular}{c} 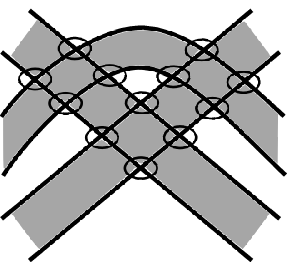 \end{tabular} \\ & (A) & & & (B) & \\
\begin{tabular}{c}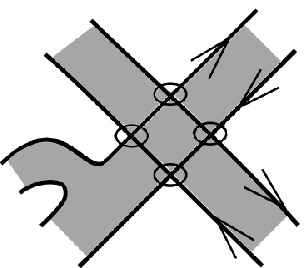 \end{tabular} & $\leftrightarrow$ & \begin{tabular}{c}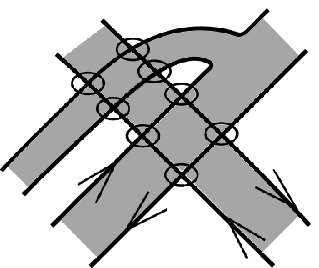\end{tabular} & 
\begin{tabular}{c}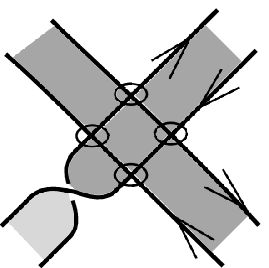 \end{tabular} & $\leftrightarrow$ & \begin{tabular}{c}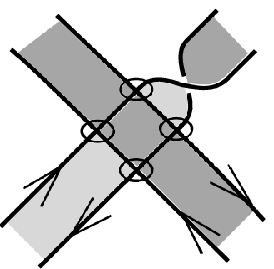\end{tabular} \\   
& (C) & & & (D) & \\
\end{tabular}
\caption{Moves showing the effect of altering a projector and screen $\chi:\Xi \to X$ on $\chi(F)$. See accompanying text in Section \ref{sec_method2} on how to interpret them.} \label{fig_proj_alt}
\end{figure}

\begin{figure}[p]
\makebox[\textwidth][c]{
\begin{tabular}{|cc|} \hline &\\ 
\begin{tabular}{c} \def\svgwidth{3in} 
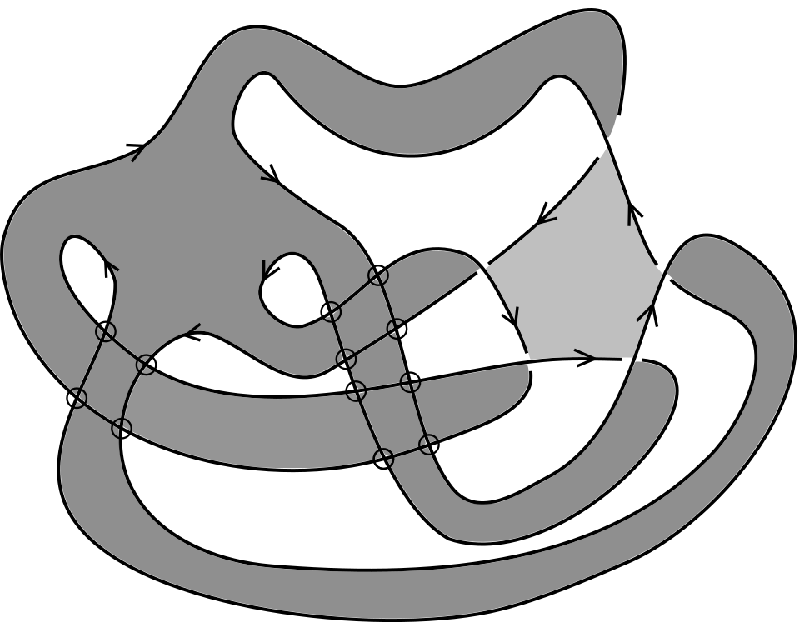\end{tabular} & \begin{tabular}{c} \def\svgwidth{3in} 
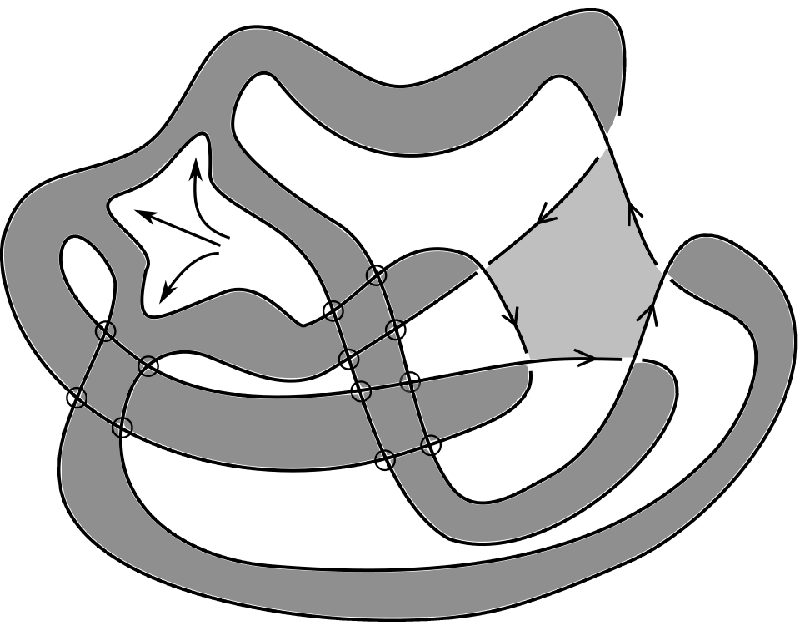\end{tabular}  \\  (1) & (2) \\ & \\
\begin{tabular}{c} \def\svgwidth{3in} 
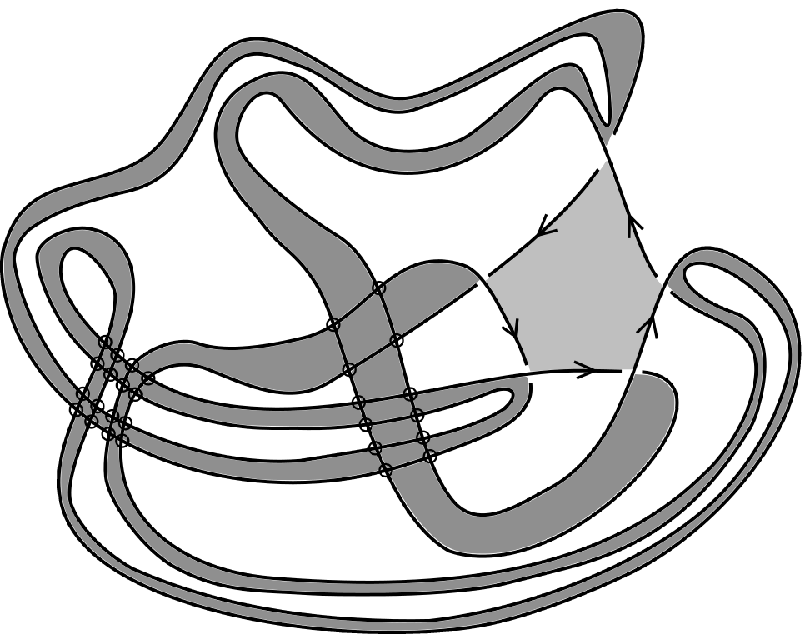\end{tabular} & \begin{tabular}{c} \def\svgwidth{3in} 
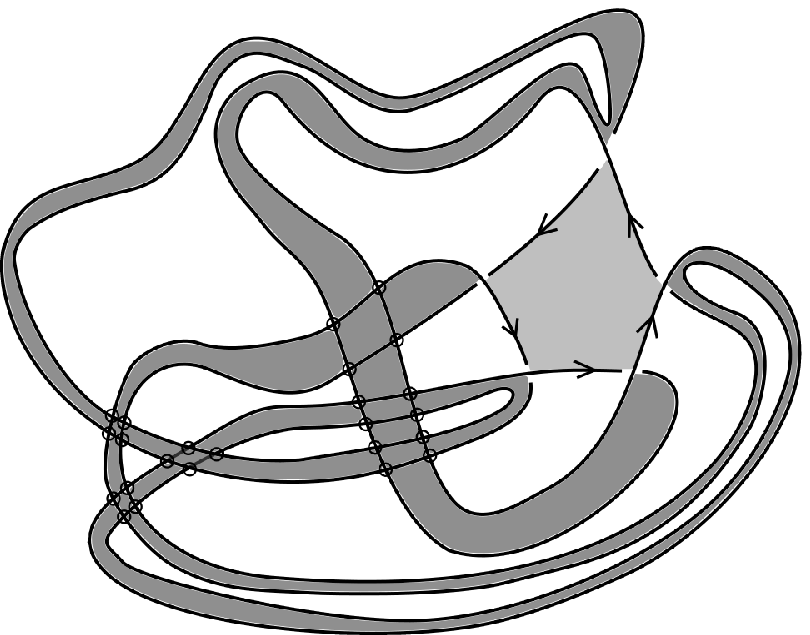\end{tabular}  \\  (3) & (4) \\ & \\
\begin{tabular}{c} \def\svgwidth{3in} 
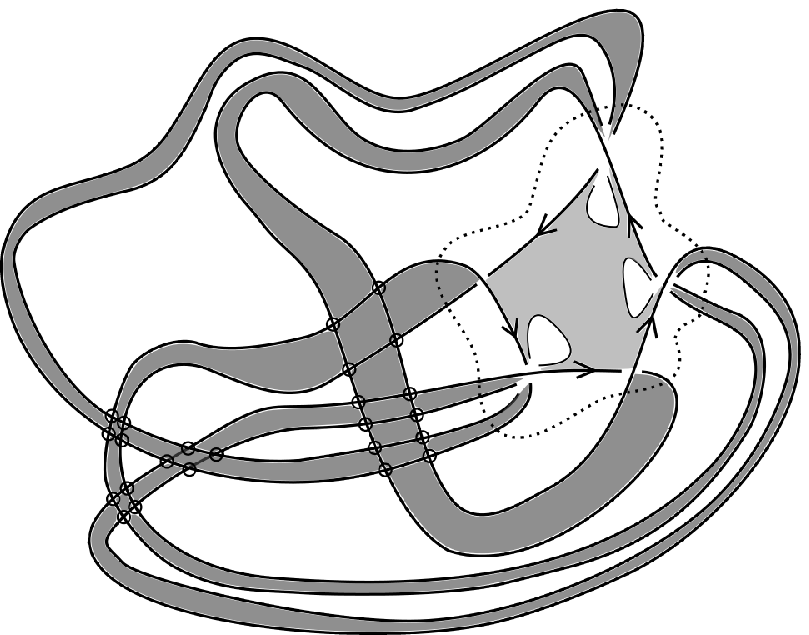\end{tabular} & \begin{tabular}{c} \def\svgwidth{2in} 
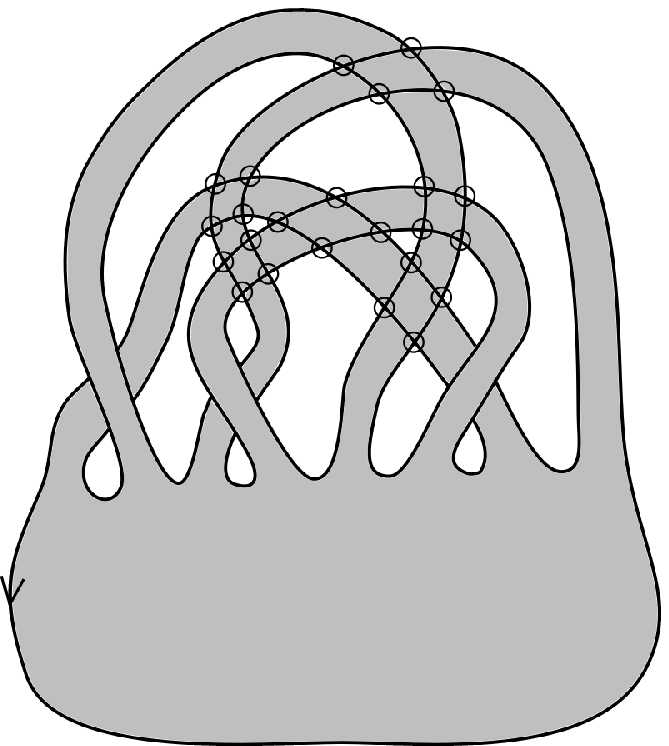\end{tabular}  \\  (5) & (6) \\ & \\ \hline
\end{tabular}}
\caption{Deforming a virtual Seifert surface of 5.2025 to a virtual band presentation.} \label{fig_5pt2025_deform}
\end{figure}

\textbf{Example (5.2025):} Here we find a virtual band presentation for 5.2025 from the virtual Seifert surface $(F;\chi:\Xi \to X)$ constructed in Figure \ref{fig_5pt2025_screen}. The steps used are given in Figure \ref{fig_5pt2025_deform}.  Step (2) can be obtained from step (1) by altering $\chi:\Xi \to X$. Indeed, the disc $d_2$ is replaced in the handle decomposition of $\Xi$ with three fattened $Y$-shaped regions. The vertex of each $Y$ is an embedded disc in the screen. Next, step (3) is obtained from step (2) by three moves of type (C) and some re-positioning of the bands. Performing move (A) together with move (B) gives the virtual Seifert surface seen in step (4). The dotted curve in step (5) encloses a region over which the projector $\chi$ is one-to-one. Thus, we may perform moves from Figure \ref{fig_proj_oto} (see bottom left and right) on the virtual Seifert surface in step (4). Lastly, step (6) is obtained from step (5) via isotopies of $F$ and moves of type (D). In this example, every half-twist of a band can be seen to cancel with an opposite half-twist. Thus, none of the bands in the virtual band presentation have any curls (such as in Figure \ref{fig_proj_oto}, top left and right).

\section{Computing Signatures and Alexander Polynomials} \label{sec_comp}

Given an AC Gauss diagram of an AC knot, a virtual Seifert surface can be constructed and deformed to a virtual band presentation. In this section, it is shown how to use virtual band presentations to compute the Alexander polynomials and directed signature functions from \cite{bcg2} and \cite{acpaper}. We begin with a survey of these results and their relation to virtual knot concordance.

\subsection{Linking numbers, Seifert matrices, and invariants} \label{sec_link} Let $\mathscr{J}$ be an oriented knot in a thickened closed surface $\Sigma \times I$. Using a Mayer-Vietoris argument, it follows that $H_1(\Sigma \times I\smallsetminus \mathscr{J}, \Sigma \times \{1\}) \cong \mathbb{Z}$ and is generated by a meridian $\mu$ of $\mathscr{J}$ (see \cite{acpaper}, Proposition 7.1). If $\mathscr{K}$ is a knot in $\Sigma \times I\smallsetminus \mathscr{J}$, then $[\mathscr{K}]=m \mu \in H_1(\Sigma \times I\smallsetminus \mathscr{J},\Sigma \times \{1\})$. The linking number of $\mathscr{J}$ and $\mathscr{K}$, denoted $\lk(\mathscr{J},\mathscr{K})$, is then defined to be $\lk(\mathscr{J},\mathscr{K})=m$. In contrast with the classical case, the linking number is not symmetric. However, a fundamental relation was given by Cimasoni-Turaev \cite{ct}:
\begin{eqnarray*}\label{ct_rel}
\lk(\mathscr{J},\mathscr{K})-\lk(\mathscr{K},\mathscr{J}) &=& \pi_*([\mathscr{J}])\cdot\pi_*([\mathscr{K}]),
\end{eqnarray*}
where $\pi:\Sigma \times I \to \Sigma$ is projection onto the first factor and $\cdot$ represents the algebraic intersection number. The linking number in $\Sigma \times I$ can be computed combinatorially from a link diagram $J \sqcup K$ on $\Sigma$ as follows: $\lk(J,K)$ is the number of times $J$ crosses over $K$, counted with sign (\cite{acpaper}, Section 7).
\newline
\newline
Let $\mathscr{K}$ be a knot in $\Sigma \times I$ bounding a Seifert surface $F$. Let $\alpha_1,\ldots,\alpha_{2n}$ be simple closed curves on $F$ representing a basis of $H_1(F;\mathbb{Z})$. Define a bilinear form $\beta^{\pm}:H_1(F) \times H_1(F) \to \mathbb{Z}$ by $\beta^{\pm}(x,y)=\lk(x^{\pm},y)$, where $x^{\pm}$ denotes the $\pm$-push-off of $x$ into $\Sigma \times I$. A Seifert matrix for $\beta^{\pm}$ with respect to the basis $[\alpha_1],\ldots, [\alpha_{2n}]$ is given by $V^{\pm}=(\lk(\alpha_i^{\pm},\alpha_j))$. For classical knots, $V^+=(V^-)^{\uptau}$, but this is not for all AC knots.  However, it follows from Equation \ref{ct_rel} that $V^++(V^+)^{\uptau}=V^-+(V^-)^{\uptau}$. In fact, any pair $(V^-,V^+)$ of integral matrices satisfying $\det(V^--V^+)=1$ and $V^++(V^+)^{\uptau}=V^-+(V^-)^{\uptau}$ is realizable as the Seifert pair of an AC knot (\cite{bcg2}, v2).
\newline
\newline
As in the classical case, Seifert matrices are used to define invariants of AC knots. In \cite{acpaper}, Definition 7.6 (see also Remark 7.8), the \emph{Alexander-Conway polynomial} of an AC knot $\mathscr{K} \subset \Sigma \times I$ was defined to be:
\[
\Delta_{\mathscr{K}}(t)=\det(t^{1/2}V^--t^{-1/2}V^+).
\]
This is a well-defined invariant of AC knots up to multiplication by powers of $t^{g}$, where $g$ is the virtual genus\footnote{The virtual genus is the minimum genus among all surfaces on which the virtual knot can be represented.} of $\mathscr{K}$. In \cite{bcg2}, Section 3.3, \emph{directed Alexander-Conway polynomials} and \emph{directed signature functions} were studied. They are given by:
\begin{eqnarray*}
\nabla_{\mathscr{K},F}^{\pm}(t) &=& \det(t^{1/2} V^{\pm}-t^{-1/2}(V^{\pm})^{\uptau}), \text{ and} \\
\hat{\sigma}_{\omega}^{\pm}(\mathscr{K},F) &=& \text{sig}((1-\omega)V^{\pm}+(1-\bar{\omega})(V^{\pm})^{\uptau})),
\end{eqnarray*}
respectively, where $\omega$ is a unit complex number not equal to $1$. Note that both $\nabla_{\mathscr{K},F}^{\pm}$ and $\hat{\sigma}_{\omega}^{\pm}(\mathscr{K},F)$ depend on the choice of Seifert surface.  It is important to note that different Seifert surfaces can produce different directed Alexander-Conway polynomials and directed signature functions. This dependence on $F$ is useful for computing the slice genus of homologically trivial knots in $\Sigma \times I$. 

\begin{definition}[Slice Genus] \label{defn_slice} Let $\mathscr{K}_0,\mathscr{K}_1$ be a knots in $\Sigma_0 \times I,\Sigma_1 \times I$, respectively. Then $\mathscr{K}_0$ and $\mathscr{K}_1$ are said to be \emph{concordant} if there is a compact oriented $3$-manifold $M$ and an annulus $A$ embedded in $M \times I$ such that $\Sigma_1 \sqcup -\Sigma_0 \hookrightarrow \partial M$ and $\partial A=\mathscr{K}_1 \sqcup -\mathscr{K}_0$. A knot that is concordant to the unknot in $S^2 \times I$ is said to be \emph{slice}. Any knot $\mathscr{K}$ in $\Sigma \times I$ can be connected to the unknot in $S^2 \times I$ by some compact oriented surface $S$ in some $M^3 \times I$. The smallest possible genus among all such surfaces $S$ and $3$-manifolds $M$ is called the \emph{slice genus}, denoted $\check{g}_s(\mathscr{K})$.
\end{definition}

The definition of concordance and slice genus given above is due to Turaev \cite{turaev_cobordism}. In \cite{lou_cob}, Kauffman introduced a combinatorial definition for virtual knots. It follows from the main result of Carter-Kamada-Saito \cite{CKS} that the two definitions are equivalent in the smooth category. Boden-Nagel \cite{boden_nagel} proved that a classical knot is slice in the classical sense if and only it is slice in the sense of Definition \ref{defn_slice}. Dye-Kaestner-Kauffman \cite{DKK} extended the Rasmussen invariant to virtual knots. This determines the slice genus exactly for virtual knots having all positive or all negative crossings. In \cite{bcg1}, Theorem 7, it was shown that the odd writhe, Henrich-Turaev polynomial, and writhe polynomial are all concordance invariants. An effective lower bound on the slice genus is the graded genus of Turaev \cite{turaev_cobordism}.  For computational results on the slice genus of virtual knots, see \cite{bcg2,bcg1,rush}.
\newline
\newline
It is interesting to study slice obstructions for AC knots for several reasons. All AC knots have a diagram with trivial index at each crossing and hence the odd writhe, Henrich-Turaev polynomial, and writhe polynomial are all vanishing. Thus, finer obstructions are needed to determine the slice genus of AC knots. In \cite{bcg2}, Theorem 3.6, it was proved that if $\mathscr{K}$ bounds a Seifert surface $F$ in some $\Sigma \times I$ and $\omega \in S^1 \smallsetminus \{1\}$ satisfies $\nabla^{\pm}_{\mathscr{K},F}(\omega) \ne 0$ then:
\[
\frac{|\hat{\sigma}^{\pm}_{\omega}(\mathscr{K},F)|}{2} \le \check{g}_s(\mathscr{K}).
\]
Furthermore, if $\mathscr{K}$ is a slice AC knot, then there are polynomials $f^{\pm}(t) \in \mathbb{Z}[t]$ such that $\nabla^{\pm}_{\mathscr{K},F}(t)=f^{\pm}(t)\cdot f^{\pm}(t^{-1})$. These tools, together with the graded genus and the virtual Rasmussen invariant are sufficient to determine the slice status of all 76 AC knots having at most 6 crossings. Another reason AC knots are interesting is that any concordance invariant of AC knots lifts to a concordance invariant of virtual knots. This follows from the fact that there is a map from virtual knots to AC knots that preserves the concordance relation (see \cite{bcg2}, Section 2.3 and Theorem 2.9.) The map is known as Turaev's coverings of knots \cite{turaev_cobordism}, or equivalently, Manturov's parity projection \cite{manturov}. For a proof of the equivalence of parity projection and coverings, see \cite{bcg2}, Lemma 2.6.
 
\subsection{Computing with virtual Seifert surfaces} \label{sec_compute_last} Here we show how to compute $\Delta_{\mathscr{K}}(t)$, $\nabla_{\mathscr{K},F}^{\pm}(t)$, and $\hat{\sigma}^{\pm}_{\omega}(\mathscr{K},F)$ from a virtual band presentation. To do this, we must first recall the virtual linking number $\vlk(\upsilon_1,\upsilon_2)$. For a two component oriented virtual link $\upsilon_1 \sqcup \upsilon_2$, $\vlk(\upsilon_1,\upsilon_2)$ is the sum of the classical crossings signs where $\upsilon_1$ crosses over $\upsilon_{2}$. 
\newline
\newline
Now, let $F_{\tau}$ be a virtual band presentation with underlying tangle $\tau$ and $\upsilon$ the virtual knot diagram virtually bounding $F_{\tau}$. By Theorem \ref{thm_vss_band}, there is a virtual Seifert surface $(F;\chi:\Xi \to X)$ with $\chi(F)=F_{\tau}$. Let $\alpha_1,\ldots,\alpha_{2n}$ be a canonical system of simple closed curves forming a basis of $H_1(F)$. Such a system of curves, for example, can be found by closing up the ends of the virtual tangle $\tau$ in the natural way. The Seifert matrices are then given by $V^{\pm}=(\lk(\alpha_i^{\pm},\alpha_j))$. Recall that $\lk(\alpha_i^{\pm},\alpha_j)$ is the sum of classical crossing signs where $\alpha_i^{\pm}$ crosses over $\alpha_j$. Using the projector $\chi$, we obtain a virtual link diagram $\chi(\alpha_i^{\pm} \sqcup \alpha_j)$ (see discussion in Section \ref{sec_vss_gen}). By the above definition of $\vlk$, we see that $\lk(\alpha_i^{\pm},\alpha_j)=\vlk(\chi(\alpha_i^{\pm} \sqcup \alpha_j))$. We record this observation in the following theorem.

\begin{theorem} With the notations as above, $V^{\pm}=(\vlk(\chi(\alpha_i^{\pm}\sqcup\alpha_j)))$.
\end{theorem}

\begin{remark} The direction of the $\pm$-push-off is determined by the orientation of the disc $B$ defining $F_{\tau}$. The convention used here matches the convention used in \cite{bcg2}, Section 4.1. If $B$ is oriented counterclockwise, the $+$ direction is away from the reader (i.e. into the page) and the $-$ direction is towards the reader. The direction of the push-offs is opposite if $B$ is clockwise oriented.
\end{remark}

\textbf{Example (4.99):} The AC knot 4.99 is slice. A virtual band presentation is given in Figure \ref{fig_4pt99_summ}. The image under $\chi$ of an ordered basis $\{\alpha_1,\alpha_2\}$ in $X$ is also drawn. Relative to this basis, we obtain the following Seifert matrices $V^{\pm}$.
\[
V^+=\left[\begin{array}{cc} 
1 & 0  \\ 
0 & -1  \\
\end{array} \right],
\,\,\,\,
V^-=\left[\begin{array}{cc} 
1 & 1  \\ 
-1 & -1  \\
\end{array} \right],
\]
From this we conclude that $\Delta_{\mathscr{K}}(t)=2-(1/t)$, $\nabla_{\mathscr{K},F}^{-}(t)=4$, and $\nabla_{\mathscr{K},F}^{+}(t)=2-t-(1/t)$. Moreover, $\hat{\sigma}_{\omega}^{\pm}(\mathscr{K},F)=0$ for all $\omega \in S^1\smallsetminus 1$. Note also that $\nabla^{\pm}_{\mathscr{K},F}(t)$ factor in the expected way for slice knots. In particular:
\[
\nabla_{\mathscr{K},F}^+(t)=2-t-\frac{1}{t}=(1-t)\left(1-\frac{1}{t}\right).
\] 
\textbf{Example (5.2025):} The AC knot 5.2025 is also slice. For the virtual band presentation in Figure \ref{fig_5pt2025}, we obtain the following Seifert matrices for the ordered basis $\{\alpha_1,\alpha_2,\alpha_3,\alpha_4\}$.
\[
V^+=\left[\begin{array}{cccc} 
0 & 0 & 0 & 0 \\
0 & 0 & 0 & 0 \\
-1 & 1 & 0 & 0 \\ 
0 & -1  & 0 & 0 \\
\end{array} \right],
\,\,\,\,
V^-=\left[\begin{array}{cccc}
0 & 1 & 0 & 1 \\
-1 & 0  & 0 & 0 \\ 
-1 & 1  & 0 & 1 \\ 
-1 & -1  & -1 & 0 \\ 
\end{array}
\right]
\]
It follows that the Alexander-Conway polynomial is given by $\Delta_{\mathscr{K}}(t)=t$. Moreover, $\hat{\sigma}_{\omega}^{-}(K,F)=0$ for all $\omega \ne 1$. The directed Alexander-Conway polynomials are $\nabla^{+}_{\mathscr{K},F}(t)=1$ and $\nabla^{-}_{\mathscr{K},F}(t)=1$.   
\newline
\newline
\textbf{Example (6.87548):} Figure \ref{fig_6pt87548}, bottom right, gives the image under $\chi$ of the ordered basis $\{\alpha_1,\alpha_2,\alpha_3,\alpha_4\}$ for $H_1(F)$. Relative to this basis, the Seifert matrices $V^{\pm}$ are:
\[
V^+=\left[\begin{array}{cccc} 
-1 & 0 & 0 & 0 \\
-1 & 1 & 1 & 0 \\
0 & -1 & 0 & 1 \\ 
0 & 0  & 0 & 1 \\
\end{array} \right],
\,\,\,\,
V^-=\left[\begin{array}{cccc}
-1 & -1 & 0 & 0 \\
0 & 1  & 0 & 0 \\ 
0 & 0  & 0 & 0 \\ 
0 & 0  & 1 & 1 \\ 
\end{array}
\right]
\]
Using these Seifert matrices, we compute the Alexander-Conway and directed Alexander-Conway polynomials as follows:
\begin{eqnarray*}                  
\Delta_{\mathscr{K}}(t) &=& 2-t+\frac{1}{t}-\frac{1}{t^2}, \\
\nabla_{\mathscr{K},F}^+(t) &=& 5-t-t^2-\frac{1}{t}-\frac{1}{t^2}, \text{ and} \\
\nabla_{\mathscr{K},F}^-(t) &=& 3-t-\frac{1}{t}.\\
\end{eqnarray*}
Note that the computation of $\Delta_{\mathscr{K}}(t)$ agrees with the computation in \cite{bcg2}, Table 1, up to multiplication by $t^g$, where $g=1$ is the virtual genus of $6.87548$. Furthermore, observe that $\hat{\sigma}^{\pm}_{\omega}(\mathscr{K},F)=0$ for all $\omega \in S^1\smallsetminus \{1\}$. However, $\nabla_{\mathscr{K},F}^-(t)$ does not factor as $f(t)f(t^{-1})$ for any polynomial $f(t) \in \mathbb{Z}[t]$. Thus, $6.87548$ is not slice and we conclude $\check{g}_s(6.87548) \ge 1$. An explicit genus one cobordism from 6.87548 to the unknot was found in \cite{bcg1}. It follows that $\check{g}_s(6.87548)=1$.

\section{Canonical Seifert Genus} \label{sec_canon}
Let $\mathscr{K}$ be a classical knot in $S^3$. The \emph{$3$-genus} of $\mathscr{K}$, denoted $g(\mathscr{K})$, is the minimum genus among all Seifert surfaces that $\mathscr{K}$ bounds in $S^3$. The \emph{canonical $3$-genus}, denoted $g_c(\mathscr{K})$, is the minimum genus among all Seifert surfaces obtained through Seifert's algorithm (taken over all diagrams of $\mathscr{K}$). The $3$-genus and canonical $3$-genus are not always equal. Moriah \cite{moriah} proved that their difference can be made arbitrarily large. Here we consider the genus and canonical genus for virtual knots. First, the virtual $3$-genus and the virtual canonical $3$-genus of an AC knot is defined. These definitions are then related to work of Boden-et-al.\cite{acpaper}, Kauffman \cite{lou_cob}, Stoimenow-Tchernov-Vdovina \cite{sto_canon}, and Tchernov \cite{chernov_proj}. We then prove that the virtual canonical $3$-genus of a classical knot $\mathscr{K}$ is $g_c(\mathscr{K})$.

\begin{definition}[Virtual 3-Genus, Virtual Canonical 3-Genus] Let $\upsilon$ be a virtual knot. The \emph{virtual $3$-genus} of $\upsilon$ is defined for AC knots to be the minimum genus among all Seifert surfaces for all homologically trivial representatives $\mathscr{K}$ of $\upsilon$ in some thickened surface $\Sigma \times I$ (see also \cite{acpaper}, Definition 6.3). We will denote it by $\check{g}(\upsilon)$. If $\upsilon$ is not AC, we set $\check{g}(\upsilon)=\infty$. A virtual Seifert surface is said to be \emph{canonical} if it is constructed using the virtual Seifert surface algorithm. The \emph{virtual canonical $3$-genus} is defined for AC knots to be the minimum genus among all canonical virtual Seifert surfaces of $\upsilon$. We denote this by $\check{g}_{c}(\upsilon)$. If $\upsilon$ is not AC, we set $\check{g}_{c}(\upsilon)=\infty$.
\end{definition} 

In \cite{acpaper}, Theorem 7.9, it was proved that $\text{width}(\Delta_{\mathscr{K}}(t))/2$ is a lower bound on the virtual $3$-genus, where the width is the difference between the highest and lowest degree terms in $\Delta_{\mathscr{K}}(t)$ (see also \cite{ct}, Proposition 4.1). In \cite{acpaper}, Corollary 6.5, it was shown that a minimal genus Seifert surface can always be realized in a surface $\Sigma \times I$, where the genus of $\Sigma$ is the smallest for which $\upsilon$ can be represented as a knot in $\Sigma \times I$. Thus, if $\mathscr{K}$ is a classical knot in $S^3$, $\check{g}(\mathscr{K})=g(\mathscr{K})$.
\newline
\newline
For classical knots $\mathscr{K} \subset S^3$, the \emph{$4$-genus} is the smallest genus among all compact, connected, oriented, and smooth surfaces in $B^4$ that are bounded by $\mathscr{K} \subset \partial B^4$. The 4-genus is denoted $g_4(\mathscr{K})$. One way to obtain a surface in $B^4$ bounded by $\mathscr{K}$ is to push any Seifert surface $F\subset S^3$ of $\mathscr{K}$ into $B^4$. Thus, $g_4(\mathscr{K}) \le g(\mathscr{K})$. It is currently unknown if $g_4(\mathscr{K})=\check{g}_s(\mathscr{K})$. By Boden-Nagel \cite{boden_nagel}, Theorem 2.8, the equality holds when $\mathscr{K}$ is slice. However, a virtual version of the inequality $g_4\le g$ holds for a different notion of genus for virtual knots that we will call the \emph{virtual canonical slice genus}. 
\newline

The virtual canonical slice genus is defined as follows. Let $\upsilon$ be a virtual knot diagram (not necessarily AC), $\Sigma$ its Carter surface, $K$ the knot diagram on $\Sigma$, and $\mathscr{K}$ the corresponding knot in $\Sigma \times I$. Let $s_1, \ldots,s_p$ be the disjoint Seifert cycles on $\Sigma$. We now construct a $3$-manifold $M$ and a surface $S$ bounded by $\mathscr{K}$ in $M \times I$. Identify $\Sigma$ with $\Sigma \times 1$ in $\Sigma \times I$. For $1 \le i \le p$, attach a $2$-handle $T_i \approx D^2 \times I$ along a copy of $s_i$ in $\Sigma \times 0$. Let $M$ be the resulting $3$-manifold. Then each $s_i$ on $\Sigma \times 1$ bounds a disc $D_i$ in $M$ consisting of the annulus $s_i \times I$ and the core of $T_i$. Thus $\mathscr{K}$ bounds a surface $S$ in $M \times I$ consisting of the discs $D_i$ and half-twisted bands attached along the crossings of $K$. The \emph{virtual canonical slice genus} of $\upsilon$ is the minimum genus among all surfaces $S$ constructed as above, taken over all diagrams of $\upsilon$. We denote this by $\check{g}_{s,c}(\upsilon)$. Clearly, $\check{g}_{s}(\upsilon) \le \check{g}_{s,c}(\upsilon)$.
\newline

The surfaces $S \subset M \times I$ constructed from a virtual knot diagram above have appeared previously in the literature under (at least) two different guises. In \cite{lou_cob}, Kauffman introduced a combinatorial object called a \emph{virtual surface in the $4$-ball}. The surface $S$ can be interpreted as a topological realization in the $4$-manifold $M \times I$  of this combinatorial object. Virtual surfaces in the $4$-ball were used in the extension of the Rasmussen invariant to virtual knots by Dye-Kaestner-Kauffman \cite{DKK}. 
\newline

Another point of view on the surfaces $S \subset M \times I$ is due to Stoimenow-Tchernov-Vdovina \cite{sto_canon}. To each Gauss diagram $D$ of a virtual knot $\upsilon$, a surface $S_D$ may be constructed as follows. Thicken the circle of $D$ to an annulus and then attach bands along the arrows of $D$ so that the resulting compact surface is orientable. By Figure \ref{fig_seif_cross}, adding bands along the arrows of $D$ gives a surface that has the same number of boundary components as the number of Seifert cycles of $\upsilon$. Set $S_D$ to be the closed surface obtained by attaching discs to the boundary components. Observe that $S$ and $S_D$ have the same genus:  $(n-p+1)/2$, where $p$ is the number of Seifert cycles and $n$ is the number of arrows of $D$.
\newline

Consequently, $\check{g}_{s,c}(\upsilon)$ can also be defined as the minimum genus of the surfaces $S_D$ taken over all Gauss diagrams $D$ of $\upsilon$. This was studied by Stoimenow-Tchernov-Vdovina \cite{sto_canon}, where it was referred to as the virtual canonical genus. In \cite{chernov_proj}, Theorem 2.2, Tchernov proved that for classical knot $\mathscr{K}$, $\check{g}_{s,c}(\mathscr{K})=g_c(\mathscr{K})$.  

\begin{lemma}\label{thm_canon} The genera $\check{g}_c$, $\check{g}_{s,c}$, $\check{g}$, and $\check{g}_s$ are related as follows.
\begin{enumerate}
\item For all virtual knots $\upsilon$, $\check{g}_s(\upsilon) \le \check{g}_{s,c}(\upsilon) \le \check{g}_c(\upsilon)$.
\item There are virtual knots $\upsilon$ such that $\check{g}_{s,c}(\upsilon)<\check{g}_{c}(\upsilon)<\infty$. 
\item There are virtual knots $\upsilon$ such that $\check{g}_{s}(\upsilon_2)<\check{g}_{s,c}(\upsilon_2)$. 
\item In general, $\check{g}_{s,c}$ is neither an upper bound nor a lower bound for $\check{g}$.
\end{enumerate}
\end{lemma}
\begin{proof} For (1), let $(F,\chi:\Xi \to X)$ be a canonical virtual Seifert surface for an AC knot diagram $\upsilon$. We construct a Gauss diagram $D$ of $\upsilon$ as follows. The circle of $D$ is identified with $\partial F$. At each of the half-twisted bands of $F$, draw an arrow in $F$ pointing from the over-crossing arc straight down to the under-crossing arc (see Figure \ref{fig_canon_ex}, where the arrows have been moved slightly so that they are visible). Observe that the set of embedded arrows are disjoint in $F$. Furthermore, $\partial F$ together with the embedded arrows forms a Gauss diagram $D$ of $\upsilon$. Let $V(D)$ be a regular neighborhood of $D$ in $F$. The surface obtained by attaching discs to $\partial V(D)$ is precisely a surface of the form $S_D$ as described above. 
\newline

Now, $F$ can also obtained from $V(D)$ by gluing surfaces with boundary to $V(D)$. These surfaces might, for example, be annulus or a pair of pants. If an annulus is glued to two boundary components of $V(D)$, then the genus has to increase. Indeed, you are adding a handle to $V(D)$. The same thing goes for a pair of pants. Thus we observe that $S_D$ is the smallest possible genus surface that can be created by gluing surfaces with boundary to $V(D)$. Hence, the genus of $S_D$ is at most the genus of $F$. Thus:
\[
\check{g}_{s,c}(\upsilon)=\min\{\text{genus}(S_E): E \text{ is a Gauss diagram of } \upsilon\} \le \text{genus}(S_D) \le \text{genus}(F).
\]
It follows that $\check{g}_{s,c}(\upsilon) \le \check{g}_{c}(\upsilon)$, as required. The second inequality in (1) was discussed previously.

From Figure \ref{fig_6pt87548}, we have a virtual Seifert surface of genus $2$ for 6.87548. Moreover, $\text{width}(\Delta_{\mathscr{K}}(t))=3$. Thus, $\check{g}_c(6.87548)=\check{g}(6.87548)=2$. Again consulting Figure \ref{fig_6pt87548}, it follows that $\check{g}_{s,c}(6.87548)=1=\check{g}_s(6.87548)$ (see Section \ref{sec_compute_last}). Thus $\check{g}_{s,c}(6.87548)<\check{g}_c(6.87548)=\check{g}(6.87548)<\infty$. Thus (2) holds and $\check{g}_{s,c}$ is not an upper bound for $\check{g}$. Now, by Moriah's theorem \cite{moriah} (see corollary to the main theorem in Section 3 therein), there is a classical knot $\mathscr{K}$ such that $g(\mathscr{K})<g_c(\mathscr{K})$. Then by the preceding discussion, $\check{g}(\mathscr{K})=g(\mathscr{K})<g_c(\mathscr{K})=\check{g}_{s,c}(\mathscr{K})$, where we recall that the equalities hold because $\mathscr{K}$ is classical. Thus, $\check{g}_{s,c}$ is also not a lower bound for $\check{g}$ and (4) follows. For (3), suppose that $\mathscr{K}$ is a non-trivial slice classical knot. Then we have: 
\[
\check{g}_{s}(\mathscr{K})=g_4(\mathscr{K}) =0<1\le g_c(\mathscr{K})=\check{g}_{s,c}(\mathscr{K}), 
\]
where all equalities hold because $\mathscr{K}$ is assumed to be classical and slice.
\end{proof}

\begin{figure}[htb]
\begin{tabular}{|c|} \hline \\
$
\xymatrix{
\begin{array}{c} 
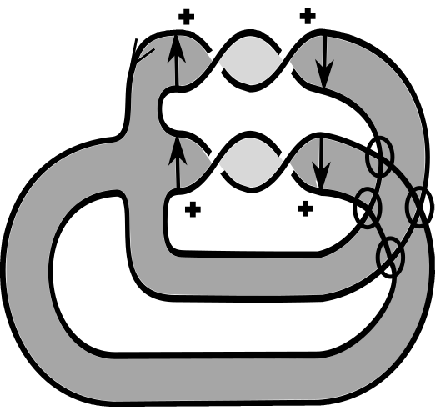\end{array} 
\ar[r] &
\begin{array}{c}
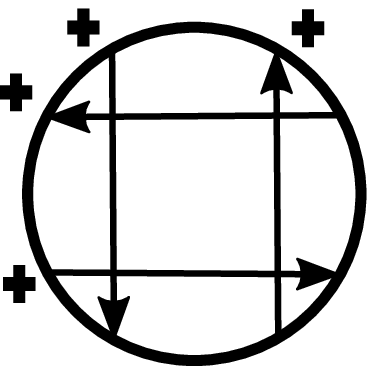\end{array}
} 
$ 
\\ \hline
\end{tabular}
\caption{Obtaining a Gauss diagram of an AC knot from a virtual Seifert surface.} \label{fig_canon_ex}
\end{figure}

The main theorem of this section is the following result. It states that the minimal genus surface produced by the virtual Seifert surface algorithm applied to a classical knot $\mathscr{K}$ has genus equal to the classical $3$-genus of $\mathscr{K}$, where the minimum is taken over all Gauss diagrams of virtual knots representing $\mathscr{K}$.

\begin{theorem} If $\mathscr{K}$ is a classical knot, $\check{g}_c(\mathscr{K})=g_c(\mathscr{K})$.
\end{theorem}
\begin{proof} From the definitions, it follows that $\check{g}_c(\mathscr{K}) \le g_c(\mathscr{K})$. By Lemma \ref{thm_canon}, $\check{g}_{s,c}(\mathscr{K}) \le \check{g}_{c}(\mathscr{K})$. By Tchernov \cite{chernov_proj}, Theorem 2.2, $\check{g}_{s,c}(\mathscr{K})=g_c(\mathscr{K})$. 
\end{proof}

\begin{acknowledgments} This work grew out of collaborations with H. U. Boden, R. Gaudreau, and R. Todd. The author is indebted to them for numerous helpful conversations and careful readings of previous drafts. A special thanks is owed to H. U. Boden, whose suggestions led to the significant improvement of definitions, notation, and terminology.  The author is also grateful for encouragement and advice from P. Cahn and H. A. Dye. 
\end{acknowledgments}

\bibliographystyle{plain}
\bibliography{bib_virt_seif}

\end{document}